\documentclass{amsart}

\pdfoutput=1

\usepackage{amsmath}
\usepackage{amssymb}
\usepackage{amsthm}
\usepackage{xypic}
\usepackage{graphicx}
\usepackage{color}

\newtheorem{lemma}{Lemma}[section]
\newtheorem{theorem}[lemma]{Theorem}
\newtheorem{cor}[lemma]{Corollary}
\newtheorem{prop}[lemma]{Proposition}
\newtheorem{defn}[lemma]{Definition}

\newcommand{\Z}{\mathbb{Z}}
\newcommand{\Q}{\mathbb{Q}}

\newcommand{\Ksm}[2]{K(\mathbf{#1},\mathbf{#2})}

\title[Concordance invariants from $E(-1)$]{Concordance invariants from the $E(-1)$ spectral sequence on Khovanov homology}
\author{William Ballinger}
\email{whb5@math.princeton.edu}
\address{Department of Mathematics \\ Princeton University}

\begin{document}

\maketitle

\begin{abstract}
We construct the concordance invariant coming from the $E(-1)$ spectral sequence on Khovanov homology in the same way Rasmussen's $s$ invariant comes from the Lee spectral sequence, and show that it gives a bound on the nonorientable slice genus. 
\end{abstract}

\section{Introduction}

Rasmussen's $s$ invariant \cite{rasmussen2010khovanov} is defined via the Lee perturbation of the differential on Khovanov homology: adding the Lee differential to the original Khovanov differential gives a filtered chain complex with one dimensional total homology, and the filtration degree of a generator gives the $s$ invariant. It was conjectured in \cite{dunfield2006superpolynomial} and proven in \cite{rasmussen2016some} that Khovanov homology (as well as all of the Khovanov-Rozansky $\operatorname{sl}(n)$ homologies) admits a second perturbation $d_{-1}$ of the differential with one-dimensional total homology, to which this same construction can be applied.

Specifically, adding $d_{-1}$ to the Khovanov differential gives a filtered chain complex $\bar{C}(K)$, with the filtration coming from the original homological grading. Any knot cobordism from a knot $K_0$ to another knot $K_1$ induces a map ${\bar{C}(K_0) \to \bar{C}(K_1)}$, but unlike the cobordism maps on Lee's complex this map is zero on total homology as soon as the genus of the cobordism is at least $1$. Somewhat surprisingly, however, a different map can be defined for any nonorientable cobordism $K_0 \to K_1$ that induces an isomorphism on the total homology. This leads to bounds on the nonorientable slice genus of a knot, where in what follows $t(K)$ is the knot invariant defined by the highest filtration degree of a generator of the homology of $\bar{C}(K)$:
\begin{theorem}\label{nonorientablegenusbound}
$t(K)$ is a concordance homomorphism valued in even integers. If $K \subset S^3$ bounds a (possibly nonorientable) surface $\Sigma \subset D^4$ of normal euler number $e$, then $|t(K) + e/2| \le b_1(\Sigma).$
\end{theorem}
Combining this with a similar bound coming from the knot signature $\sigma(K)$ (normalized so that $\sigma(T_{2,3}) = -2$), just as in \cite{batson2014nonorientable} and \cite{ozsvath2016unoriented}, gives the bound
\begin{equation}
|t(K) + \sigma(K)| \le 2\gamma_4(K)
\end{equation}
where $\gamma_4(K)$ is the minimal value of $b_1(\Sigma)$ over all nonorientable surfaces with boundary $K$. Just as for the $s$ invariant, defining $t$ requires a choice of coefficient field. The above property holds over any field, but in principle the value of $t$ on particular knots could depend on this choice. 

This bound is essentially identical to the nonorientable four-genus bound from \cite{ozsvath2016unoriented}, except with the invariant $t(K)$ in place of the invariant $\upsilon(K) = \Upsilon_K(1)$ used by Ozsv\'ath, Stipsicz, and Szab\'o. For many knots, including alternating knots and torus knots, it turns out that $t(K) = \upsilon(K)$, so the two invariants give the same bound on the nonorientable four-genus. In particular, the bounds on $\gamma_4(T_{p,q})$ from \cite{jabuka2018nonorientable} could be proven using $t$ instead of $\upsilon$. However, in general the two invariants do not agree.

The complex $\bar{C}(K)$ is obtained as a quotient of a larger complex $C(K)$ which carries an action of the polynomial ring $\Z[x]$ and is related to Bar-Natan's universal Khovanov homology in the same way as $\bar{C}$ is to Khovanov homology. Using the module structure of $C(K)$ allows us to define stronger invariants: each integer $T_n(K)$ given by the maximal filtration degree of an element representing $(2x)^n$ in homology is a concordance invariant, but unlike $t(K)$ these are not concordance homomorphisms. These give potentially stronger bounds on the slice genus than $s$ and $t$.
\begin{theorem}\label{orientedgenusbound}
Each $T_i(K)$ is a negative, even integer and a concordance invariant. Furthermore, $T_{i+1}(K) \ge T_i(K)$, and $T_i(K) = 0$ for $i \ge g_4(K)$.
\end{theorem}
These invariants have similar properties to, but are not identical with, the invariants $V_i$ from knot Floer homology \cite{ni2015cosmetic}. 

Many of the constructions in this paper should also work in the context of the more general Khovanov-Rozansky $\operatorname{sl}(n)$ and HOMFLY homologies. In that setting, an analogue of Theorem~\ref{orientedgenusbound} likely holds, but Theorem~\ref{nonorientablegenusbound} is unique to the $\operatorname{sl}(2)$ case since for $n > 2$ the $\operatorname{sl}(n)$ homology depends on the orientation of a link.

\subsection*{Acknowledgments}

I would like to thank Zolt\'an Szab\'o for helpful conversations and feedback on this paper.

\section{Preliminaries on matrix factorizations}

Throughout this paper, all rings and modules will be bigraded, with the first grading called the internal grading and the second grading called the filtration grading. All rings will be concentrated in negative, even internal grading and zero filtration grading. Given a module $M$ and integers $i,j$, write $M\{i,j\}$ for a copy of $M$ with the gradings shifted so that a homogeneous element of degree $(a,b)$ in $M$ has degree $(a+i,b+j)$ in $M\{i,j\}$, and $M\{i\}$ for $M\{i,0\}$. When tensoring and shifting the gradings of linear maps, the internal grading will be used in the Koszul sign rule. When relating the gradings here to the standard $q,h$ gradings on Khovanov homology, the internal grading is $q - 3h$ and the filtration grading is $h$.

\begin{defn}
Given a ring $R$ and homogeneous element $w \in R$ of degree $(2k,0)$ with $k$ odd, a matrix factorization of $w$ is an $R$-module $C$ together with a  $R$-linear map $d: C \to C$, homogeneous of degree $(k,0)$, such that $d^2 = w$.
\end{defn}
The restriction that $k$ be odd is only to avoid needing to introduce an extra $\Z/2\Z$ grading to track signs. In every matrix factorization considered here $k$ will be $-3$. 

In constructing the knot invariants, planar tangles will be assigned matrix factorizations as above but general tangles will be assigned a more flexible object in which the differential is only filtered with respect to the second grading. When $w = 0$, this object is called a multicomplex in \cite{livernet2020spectral}; the Postnikov systems from \cite{khovanov2007virtual} are also closely related.
\begin{defn}
Given $(R,w)$ as above, a multifactorization of $w$ is an $R$-module $C$ and a sequence of $R$-linear maps $d_i: C \to C$ for $i \ge 0$, homogeneous of degree $(k,i)$, such that the sum $D = \sum_{i = 0}^\infty d_i$ satisfies $D^2 = w$.
\end{defn}
If $(C,D)$ is a multifactorization, then $(C,d_0)$ is a matrix factorization called the vertical factorization of $C$. Viewing $C$ as a filtered matrix factorization, the vertical factorization is essentially the associated graded factorization.

\begin{defn}
Given two multifactorizations $(C,D)$ and $(C',D')$, a chain map $F$ between them is a sequence of $R$-linear maps $f_i: C \to C'$ for $i \ge 0$, homogeneous of degree $(0,i)$, such that the sum $F = \sum_{i = 0}^\infty f_i$ satisfies $FD = D'F$.
\end{defn}
Again, the first term $f_0$ of a chain map $F$ will be a chain map between the vertical factorizations.

\subsection{Homotopy equivalences of multifactorizations}

The factorizations associated to a tangle defined in the next section will not be filtered chain homotopy equivalent in the usual sense. Roughly, this is because it is $d_1$, not $d_0$, that gives the original Khovanov differential. Instead, we need a slightly weaker notion.
\begin{defn}
Given an integer $n$ and two chain maps of multifactorizations $F, G: (C,D) \to (C',D')$, an $n$-homotopy between them is a sequence of $R$-linear maps $h_i: C \to C'$ for $i \ge -n$, homogeneous of degree $(-k,i)$, such that the sum $H = \sum_{i = -n}^\infty h_i$ satisfies $F - G = HD + D'H$
\end{defn}
Note that $F$ and $G$ are still required to be filtered; only $H$ is permitted to lower the filtration degree by $n$. This means that picking $F$ and $H$ as above and defining $G = F - HD - D'H$ does not always result in a new chain map $n$-homotopic to $F$. Compositions of $n$-homotopic maps are still $n$-homotopic, so $n$-homotopy classes of chain maps between multifactorizations form a category. 

When $w = 0$, a multifactorization a filtered chain complex, so has an associated spectral sequence. The following fact will not be needed for the results here, but helps motivate the definition of an $n$-homotopy:
\begin{prop}
If $F$ and $G$ are $n$-homotopic maps between multifactorizations of $w = 0$, they induce the same map on the $E^r$ page of the associated spectral sequence for $r > n$.
\end{prop}
In particular, the $1$-homotopy equivalence classes of the factorizations defined here will be tangle invariants, so for a knot the $E^2$ page and all higher pages will be knot invariants. 

The main reason for working with multifactorizations instead of more flexible filtered factorizations in this paper is a version of the homological perturbation lemma, which lifts homotopy equivalences between vertical factorizations to homotopy equivalences between multifactorizations. Notation and proofs in the rest of this section are adapted from \cite{crainic2004perturbation}
\begin{defn}
Let $(C,D)$ and $(C',D')$ be multifactorizations. Then a special deformation retract from $C$ to $C'$ consists of chain maps $P: C \to C'$ and $I: C' \to C$ such that $PI = 1$, and a $0$-homotopy $H$ between $1$ and $IP$ such that $HI = 0$, $PH = 0$, and $H^2 = 0$. 
\end{defn}
\begin{prop}\label{perturb}
Let $(C,D)$ be a multifactorization of $w$, and suppose that we are given a matrix factorization $(C',d_0')$ of $w$ and maps $p_0,i_0,$ and $h_0$ forming a special deformation retract from the vertical factorization $(C,d_0)$ to $(C',d_0')$. Then if both $C$ and $C'$ are finitely supported in the filtration grading, $d_0'$ is the first term of a differential $D'$ making $(C',D')$ into a multifactorization and $p_0,i_0,$ and $h_0$ are the first terms in maps $P,I,$ and $H$ forming a special deformation retract from $(C,D)$ to $(C',D')$.
\end{prop}
\begin{proof}
Let $D_1 = D - d_0 = \sum_{i > i} d_i$, and let $A = \left( \sum_{i=0}^\infty (D_1 h_0)^i \right) D_1$, where this sum really is finite since $C$ has finite support in the filtration grading, $h_0$ preserves the filtration grading, and $D_1$ increases it. Then take $D' = d_0 + p_0Ai_0$, $P = p_0 + p_0 A h_0$, $I = i_0 + h_0 A i_0$, and finally $H = h_0 + h_0 A h_0$. These form a special deformation retract from $(C,D)$ to $(C',D')$. Since $A$ strictly increases the filtration grading, the components of $D',P,I,$ and $H$ that preserve the filtration grading are $d_0', p_0, i_0,$ and $h_0$, respectively. 
\end{proof}

Multifactorizations also allow for inductive, degree-by-degree modification of the differentials:
\begin{lemma}\label{dmodify}
Suppose that $(C,D)$ and $(C,D')$ are two multifactorizations of $w$ on the same underlying module with finite support in the filtration grading. If $d_i = d_i'$ for all $i < n$, then the difference $d_n - d_n'$ is a chain map ${(C,d_0) \to (C,d_0)\{-k\}}$. If this difference is furthermore $0$-homotopic to the $0$ map (on the vertical complex), then there is a third differential $D''$ such that $(C,D')$ and $(C,D'')$ are isomorphic multifactorizations and $d_i = d_i''$ for all $i < n+1$.
\end{lemma}
\begin{proof}
The statement that the difference $d_n-d_n'$ is a chain map follows from the fact that, in the degree $n$ part of the squares of $D$ and $D'$, all terms $d_id_j$ with $0 < i,j < n$ are equal to the corresponding term $d_i'd_j'$, so since both of these squares are zero $d_0d_n + d_nd_0 = d_0'd_n' + d_n'd_0'$. Rearranging this equation and using the fact that $d_0 = d_0'$ gives that $d_0(d_n-d_n') = -(d_n-d_n')d_0$. Since $k$ is odd, the differential on $(C,d_0)\{-k\}$ is $-d_0$, so  $d_n-d_n'$ is a chain map $(C,d_0) \to (C,d_0)\{-k\}$.

Now, suppose that $h$ is a $0$-nullhomotopy of $d_n-d_n'$, so $h$ is a homogeneous map of degree $(0,n)$ satisfying $hd_0 - d_0h = d_n - d_n'$. In principle $h$ could have terms of other degrees, but since both $d_0$ and $d_n-d_n'$ are homogeneous these can be taken to be zero. The above formula has the term $-d_0h$ instead of $d_0h$ the differential on the target $(C,d_0)\{-k\}$ is $-d_0$.

Let $D''$ be defined by the infinite sum
\begin{equation*}
D'' = \sum_{j=0}^\infty (D'+hD')(-h)^j
\end{equation*}
Since $h$ raises the filtration degree by $n$, only finitely many terms of this sum contribute to each $d_i''$. For $i < n$, only the term $D'$ contributes, so $d_i'' = d_i' = d_i$. For $i = n$, $d_n'' = d_n' + hd_0' - d_0'h = d_n$, as desired. Finally, $1 + h$ is a chain map $(C,D') \to (C,D'')$ that is upper unitriangular with respect to a filtered basis, so is invertible. 
\end{proof}

\subsection{Koszul factorizations and maps between them}

In the factorizations associated to tangles, the vertical factorizations will be Koszul matrix factorizations as defined in \cite{khovanov2004matrix}. Given homogeneous elements $a,b \in R$ of degrees $2i,2j$ with $ab = w$, the length-$1$ matrix factorization $K(a,b)$ is defined by this diagram:
\begin{equation}
R\{i-j\} \xrightarrow{a} R \xrightarrow{b}
\end{equation}
where, to make diagrams less cluttered, arrows pointing to the right edge of the diagram represent maps into the object at the left edge in the same row. More generally, given sequences $\mathbf{a} = a_1,\dots,a_n$ and $\mathbf{b} = b_1,\dots,b_n$, the length-$n$ Koszul factorization is defined by a tensor product
\begin{equation}
\Ksm{a}{b} = \bigotimes_{i = 1}^n K(a_i,b_i)
\end{equation}

When simplifying tangle multifactorizations using Proposition~\ref{perturb}, we need a supply of homotopy equivalences of Koszul factorizations. These are all adapted from 2.2-2.3 of \cite{khovanov2007virtual}.
First, a change of basis in a Koszul factorization often results in another Koszul factorization:
\begin{prop}[2.34 and 2.35 from \cite{khovanov2007virtual}] \label{koszulbasischange}
Suppose that either $$(a_1',\dots,a_i',\dots,a_j',\dots) = (a_1,\dots,a_i,\dots,a_j+\lambda a_i,\dots)$$ and $$(b_1',\dots,b_i',\dots,b_j',\dots) = (b_1,\dots,b_i-\lambda b_j,\dots,b_j,\dots)$$
or $(a_i') = (a_i)$ and $$(b_1',\dots,b_i',\dots,b_j',\dots) = (b_1,\dots,b_i-\lambda a_j,\dots,b_j + \lambda a_i)$$ for some $\lambda \in R$. Then $K((a_i),(b_i))$ and $K((a_i'),(b_i')$ are isomorphic. 
\end{prop}
Second, we will need two special cases of Theorem 2.2 from \cite{khovanov2007virtual}:
\begin{prop}\label{linex}
Suppose that $R = R_0[x]$ is a polynomial ring in which $x$ has degree $-2$, and that $C=\Ksm{a}{b}$ is a Koszul factorization in which $a_n = x$. Then, as a factorization over $R_0$, $C$ has a special deformation retract to the factorization $\Ksm{\bar{a}}{\bar{b}}$, where
\begin{align*}
\mathbf{\bar{a}} &= \pi(a_1),\dots,\pi(a_{n-1}) \\
\mathbf{\bar{b}} &= \pi(b_1),\dots,\pi(b_{n-1})
\end{align*}
and $\pi: R \to R_0$ is the homomorphism sending $x$ to $0$.
\end{prop}
\begin{prop}\label{sqex}
Suppose that $R = R_0[x]$ is a polynomial ring in which $x$ has degree $-2$, and that $C=\Ksm{a}{b}$ is a Koszul factorization in which $a_n = 0$, ${b_n = x^2 + h}$ for some $h \in R_0$ of degree $-4$, and that all other $a_i$ and $b_i$ lie in $R_0$. Then, as a factorization over $R_0$, $C$ has a special deformation retract to the sum of factorizations $\Ksm{\bar{a}}{\bar{b}}\{-1,0\} \oplus \Ksm{\bar{a}}{\bar{b}}\{1,0\}$ where $\mathbf{\bar{a}}$ and $\mathbf{\bar{b}}$ are as above.
\end{prop}

Finally, the following is a key source of homotopies between endomorphisms of Koszul complexes:
\begin{lemma}\label{koszulhomot}
Each $a_i$ and $b_i$ acts nullhomotopically on $K((a_i),(b_i))$.
\end{lemma}
\begin{proof}
Since $K((a_i),(b_i))$ is a tensor product of length-$1$ Koszul factorizations, it suffices to show that multiplication by $a$ or $b$ is nullhomotopic on the factorization
\begin{equation*}
R\{i-j\} \xrightarrow{a} R \xrightarrow{b}
\end{equation*}
Let $e$ and $f$ be generators of the two copies of $R$ in the above factorization, so ${d(e) = af}$ and ${d(f) = be}$. Then the map $h_a$ sending $f$ to $e$ and $e$ to $0$ is a nullhomotopy of $a$, and similarly the map $h_b$ sending $e$ to $f$ and $f$ to $0$ is a nullhomotopy of $b$.
\end{proof}

\section{Defining the complex}

We are now ready to define a multifactorization $C(D)$ associated to any tangle diagram with a few additional decorations. The constructions in this section are essentially identical to the $n=2$ case of \cite{rasmussen2016some}, with the main differences being working over $\Z$ instead of $\Q$ and defining the underlying ring in a way that is better adapted to unoriented tangles.

For the purposes of this section, a decorated tangle diagram $D$ consists of a tangle diagram $D_0$ together certain extra data. First, add any number of disjoint dotted arcs meeting $D_0$ only in their endpoints, requiring that the union of these arcs and $D_0$ is a connected diagram. Then label each region of the complement of $D_0$ and these arcs with a variable $x_1,\dots,x_n$ and, near each arc and each crossing of $D_0$, put a mark on one of the four adjacent edges of $D_0$.

For a decorated tangle diagram $D$, let $R(D)$ be the subring of the polynomial ring $\Z[x_1,\dots,x_n]$ over the variables assigned to regions generated by all differences $x_i-x_j$. If $e$ is an oriented edge of $D_0$, let $x_e$ be the difference $x_\ell - x_r$ where $x_\ell$ and $x_r$ are the labels of the regions to the left and right of $e$. Note that the $x_e$ give an isomorphism between $R(D)$ and the ``edge ring" from \cite{rasmussen2016some}.

The multifactorization $C(D)$ will be a matrix factorization over $R(D)$ with potential
\begin{equation}
w = \frac13 \sum_{e \in \partial D} x_e^3,
\end{equation}
where $\partial D$ is the set of edges meeting the boundary of $D$. Since $\sum_{e \in \partial D} x_e = 0$, the potential $w$ has integer coefficients despite the prefactor of $\frac13$.

\begin{figure}[h]
  \centering
  \def\svgwidth{\columnwidth}
  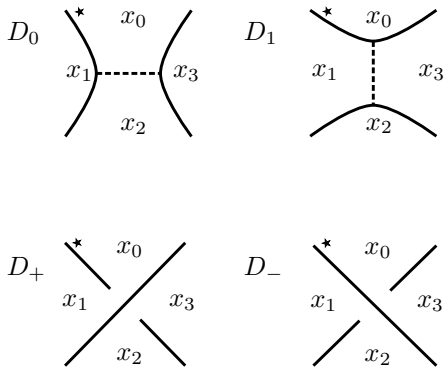
  \caption{The four basic tangles, with the marked edges labeled by $\bigstar$.}\label{basics}
\end{figure}

Consider the four decorated tangle diagrams $D_0,D_1,D_+$ and $D_-$ in Figure~\ref{basics}. The ring $R(D)$ and potential $w$ are the same in all three diagrams, and $w$ can be calculated as
\begin{align*}
w &= \frac13 \left[ (x_0-x_1)^3 + (x_1-x_2)^3 + (x_2-x_3)^3 + (x_3-x_0)^3 \right] \\
&= x_0x_1^2 - x_0^2x_1 + x_1x_2^2 - x_1^2x_2 + x_2x_3^2 - x_2^2x_3 + x_3x_0^2 - x_3^2x_0 \\
&= (x_0-x_2)(x_1^2-x_3^2) - (x_0^2-x_2^2)(x_1-x_3) \\
&= (x_0-x_2)(x_1-x_3)(x_1-x_2+x_3-x_0)
\end{align*}

For $D_0$ and $D_1$, the associated factorizations are Koszul:
\begin{align}
C(D_0) &= K(x_0-x_2,(x_1-x_3)(x_1-x_2+x_3-x_0)) \\
C(D_1) &= K(x_1-x_3,(x_0-x_2)(x_1-x_2+x_3-x_0))
\end{align}

The factorizations for $D_+$ and $D_-$ can then be defined as cones between $C(D_0)$ and $C(D_1)$. In this diagram, the diagonal maps describe a chain map ${s_{D_0 \to D_1}: D_0 \to D_1\{1\}}$.
\begin{equation}\label{elementarysaddle}
\xymatrix@C=100pt@R=40pt{
R\{1\} \ar[r]^{x_0-x_2} \ar[dr]^{-1} & R \ar[r]^{(x_1-x_3)(x_1-x_2+x_3-x_0)} \ar[dr]^{(x_1-x_2+x_3-x_0)} & \\
R\{2\} \ar[r]_{-(x_1-x_3)} & R\{1\} \ar[r]_{-(x_0-x_2)(x_1-x_2+x_3-x_0)} &
}
\end{equation}
Permuting the variables in the above diagram gives another chain map ${s_{D_1 \to D_0}: D_1 \to D_0\{1\}}$.

For $D_+$, the vertical factorization of $C(D_+)$ is $C(D_0)\{1\} \oplus C(D_1)\{-1,1\}$, the differential $d_1$ is given by $s_{D_0 \to D_1}$, and there are no higher differentials. Similarly, the vertical factorization of $C(D_-)$ is $C(D_1)\{1\} \oplus C(D_0)\{-1,1\}$ and $d_1$ is given by $s_{D_1 \to D_0}$. 

For general decorated tangle diagrams $D$, the multifactorization $C(D)$ is a tensor product of these elementary factorizations. Specifically, say that an elementary subdiagram of $D$ is a subdiagram $D_e$ containing at most one arc or crossing. Every elementary subdiagram is equivalent to one of $D_0,D_1,D_+,$ or $D_-$ up to relabeling the variables, so we can define
\begin{equation}
C(D) = \bigotimes_{D_e \subset D \text{ elementary}} C(D_e) \otimes_{R(D_e)} R(D)
\end{equation}

\subsection{Invariance of $C(D)$}

Since the complex $C(D)$ is essentially identical to the $n=2$ case of the complexes in \cite{rasmussen2016some}, existing proofs that Khovanov homology is a knot invariant go most of the way to proving that $C(D)$ is a tangle invariant. The only additional step needed is to go from isomorphisms on the $E_2$ page of a spectral sequence to honest $1$-homotopy equivalences of multifactorizations. This amounts to checking that these isomorphisms commute with higher differentials, which is automatic for the first and second Reidemeister moves and nearly so for the third. Nevertheless, we will write out the proof in detail to keep this paper self contained.

The main result will be the following:
\begin{defn}
For a tangle $T$, the boundary ring $R(\partial T)$ is the subring of $R(D)$ generated by the variables associated to inward-pointing edges meeting the boundary. This is isomorphic to the quotient of a polynomial ring on the endpoints by the sum of the endpoints.
\end{defn}
\begin{defn}
For a relatively oriented tangle $T$ and decorated diagram $D$ of $T$, let $C(T)$ be the shift $C(D)\{0,-n_-\}$ with scalars restricted to $R(T)$, where $n_-$ is the number of negative crossings in $D$.
\end{defn}
\begin{prop}\label{tangleinvariant}
Up to $1$-homotopy equivalence, $C(T)$ is an invariant of $T$.
\end{prop}

Proving this requires showing both that the decorations used to define $C(D)$ can be changed without effecting $C(D)$, and that $C(D)$ is invariant under Reidemeister moves. These are the respective objectives of the next two sections.

\subsubsection{Manipulating planar diagrams and decorations}

First, part of the decoration on a tangle diagram used to define $C(D)$ is a choice of marked edge near each crossing and dotted arc, but this only effects the differential in $C(D)$ by signs. If two diagrams differ only in marked edges, an appropriate diagonal matrix with $\pm 1$ entries on the diagonal will be an isomorphism between the associated Koszul factorizations. From now on, then, the marked edges will not be shown in diagrams unless signs are particularly relevant.

Using Proposition~\ref{koszulbasischange}, dotted arcs can be moved around the diagram:
\begin{prop}\label{slideswitch}
If $D_1$ and $D_2$ are either of the pairs of diagrams shown in Figures~\ref{arcslide}~or~\ref{arcswitch}, then $C(T_1) \cong C(T_2)$
\end{prop}
\begin{proof}
For the diagrams in Figure~\ref{arcslide}, 
\begin{align*}
&K\left( \begin{pmatrix} x_1-x_3 \\ x_3-x_5 \end{pmatrix},  \begin{pmatrix} (x_2-x_0)(x_2-x_1-x_3+x_0) \\  (x_4-x_0)(x_4-x_3-x_5+x_0) \end{pmatrix}\right) \\
\cong &K\left( \begin{pmatrix} x_1-x_5 \\ x_3-x_5 \end{pmatrix},  \begin{pmatrix} (x_2-x_0)(x_2-x_1-x_3+x_0) \\  (x_4-x_0)(x_4-x_3-x_5+x_0) - (x_2-x_0)(x_2-x_1-x_3+x_0) \end{pmatrix}\right) \\
\cong &K\left( \begin{pmatrix} x_1-x_5 \\ x_3-x_5 \end{pmatrix},  \begin{pmatrix} (x_2-x_0)(x_2-x_1-x_5+x_0) \\  (x_4-x_0)(x_4-x_3-x_5+x_0) - (x_2-x_0)(x_2-x_5-x_3+x_0) \end{pmatrix}\right) \\
= &K\left( \begin{pmatrix} x_1-x_5 \\ x_3-x_5 \end{pmatrix},  \begin{pmatrix} (x_2-x_0)(x_2-x_1-x_5+x_0) \\  (x_4-x_2)(x_4-x_3-x_5+x_2) \end{pmatrix}\right)
\end{align*}
where the first part of Proposition~\ref{koszulbasischange} is used in the first isomorphism and the second is used in the second. Similarly, for Figure~\ref{arcswitch}, 
\begin{align*}
&K\left( \begin{pmatrix} x_1-x_3 \\ x_0-x_4 \end{pmatrix},  \begin{pmatrix} (x_2-x_0)(x_2-x_1-x_3+x_0) \\  (x_5-x_3)(x_0-x_5-x_3+x_4) \end{pmatrix}\right) \\
\cong &K\left( \begin{pmatrix} x_1-x_3 \\ x_0-x_4 \end{pmatrix},  \begin{pmatrix} (x_2-x_0)(x_2-x_1-x_3+x_0) + (x_0-x_4)(x_0-x_1-x_3+x_4) \\  (x_5-x_3)(x_0-x_5-x_3+x_4) - (x_1-x_3)(x_0-x_1-x_3+x_4) \end{pmatrix}\right) \\
= &K\left( \begin{pmatrix} x_1-x_3 \\ x_0-x_4 \end{pmatrix},  \begin{pmatrix} (x_2-x_4)(x_2-x_1-x_3+x_4) \\  (x_5-x_1)(x_0-x_5-x_1+x_4) \end{pmatrix}\right) \\
\end{align*}
\end{proof}

\begin{figure}[h]
  \centering
  \def\svgwidth{\columnwidth}
  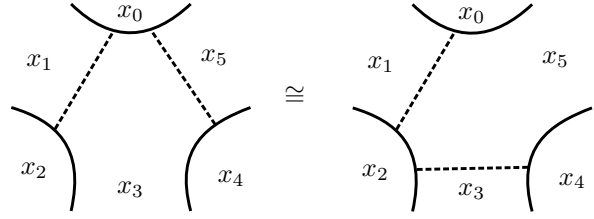
  \caption{Moving a dotted arc}\label{arcslide}
\end{figure}

\begin{figure}[h]
  \centering
  \def\svgwidth{\columnwidth}
  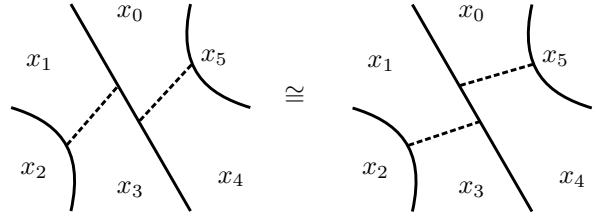
  \caption{Switching two dotted arcs}\label{arcswitch}
\end{figure}

\begin{prop}\label{arceliminate}
Suppose a tangle diagram $D$ contains a region that does not meet the boundary and is adjacent to at least one dotted arc. Then, if $D'$ is the result of removing that arc from $D$, $C(D')$ is $0$-homotopy equivalent to $C(D)$ as factorizations over $R(D')$. 
\end{prop}
\begin{proof}
Let $x_i$ be the variable labeling the region given in the problem statement, and let $x_j$ be the variable labeling the region on the other side of the dotted arc. Then $R(D) \cong R(D')[x_i-x_j]$, and the quotient map $R(D) \to R(D')$ simply replaces any occurrence of $x_i$ by $x_j$. The factorization $C(D)$ has the form
\begin{equation}
C(D) = K((a_0,\dots,a_{n-1},x_i-x_j),(b_0,\dots,b_n)).
\end{equation}
To get $C(D')$ from this, remove $a_n = x_i-x_j$ and $b_n$ and replace any other occurrence of $x_i$ by $x_j$, so by Proposition~\ref{linex} $C(D)$ and $C(D')$ are $0$-homotopy equivalent.
\end{proof}

\begin{prop}\label{loopeliminate}
If $D$ is the diagram in Figure~\ref{loopeliminatefig}, then $C(D) \cong R\{-1\} \oplus R\{1\}$ for $R = R(\partial D) = \Z[x_0-x_1]$.
\end{prop}
\begin{proof}
The ring $R(D)$ is $\Z[x_0-x_1,x_0-x_2]$, and the factorization $C(D)$ is
\begin{align*}
C(D) &= K(0,(x_2-x_1)(x_1+x_2-2x_0)) \\
&= K(0,(x_2-x_0)^2 - (x_0-x_1)^2
\end{align*}
so the result follows from Proposition~\ref{sqex}
\end{proof}

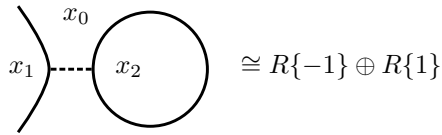
\begin{figure}[h]
  \centering
  \def\svgwidth{\columnwidth}
\begingroup%
  \makeatletter%
  \providecommand\color[2][]{%
    \errmessage{(Inkscape) Color is used for the text in Inkscape, but the package 'color.sty' is not loaded}%
    \renewcommand\color[2][]{}%
  }%
  \providecommand\transparent[1]{%
    \errmessage{(Inkscape) Transparency is used (non-zero) for the text in Inkscape, but the package 'transparent.sty' is not loaded}%
    \renewcommand\transparent[1]{}%
  }%
  \providecommand\rotatebox[2]{#2}%
  \newcommand*\fsize{\dimexpr\f@size pt\relax}%
  \newcommand*\lineheight[1]{\fontsize{\fsize}{#1\fsize}\selectfont}%
  \ifx\svgwidth\undefined%
    \setlength{\unitlength}{921.25984252bp}%
    \ifx\svgscale\undefined%
      \relax%
    \else%
      \setlength{\unitlength}{\unitlength * \real{\svgscale}}%
    \fi%
  \else%
    \setlength{\unitlength}{\svgwidth}%
  \fi%
  \global\let\svgwidth\undefined%
  \global\let\svgscale\undefined%
  \makeatother%
  \begin{picture}(1,0.13443304)%
    \lineheight{1}%
    \setlength\tabcolsep{0pt}%
    \put(0,0){\includegraphics[width=\unitlength,page=1]{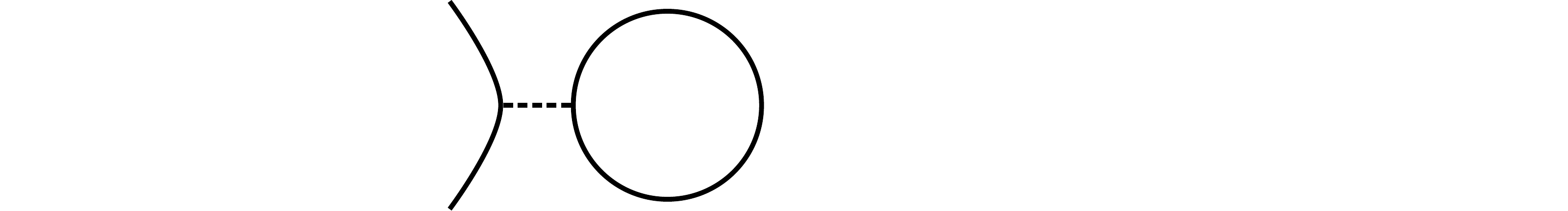}}%
    \put(0.52093052,0.06362301){\color[rgb]{0,0,0}\makebox(0,0)[lt]{\lineheight{1.25}\smash{\begin{tabular}[t]{l}$\cong R\{-1\} \oplus R\{1\}$\end{tabular}}}}%
    \put(0.33395473,0.11937259){\color[rgb]{0,0,0}\makebox(0,0)[lt]{\lineheight{1.25}\smash{\begin{tabular}[t]{l}$x_0$\end{tabular}}}}%
    \put(0.27696755,0.06319957){\color[rgb]{0,0,0}\makebox(0,0)[lt]{\lineheight{1.25}\smash{\begin{tabular}[t]{l}$x_1$\end{tabular}}}}%
    \put(0.38919741,0.06319957){\color[rgb]{0,0,0}\makebox(0,0)[lt]{\lineheight{1.25}\smash{\begin{tabular}[t]{l}$x_2$\end{tabular}}}}%
  \end{picture}%
\endgroup%

  \caption{Eliminating a loop. $R$ is the boundary ring $\Z[x_0-x_1]$.}\label{loopeliminatefig}
\end{figure}

As a first application of these moves, we can now compute the homology of the planar diagram of an unlink:
\begin{prop}\label{closedplanar}
If $D$ is a planar diagram with no boundary and $n$ closed components, then the homology of $C(D)$ is isomorphic to the module
\begin{equation*}
\Z[x_1,\dots,x_n]/(x_i^2 - x_j^2)\{n-1\} = \Z[h][x_1,\dots,x_n]/(x_i^2-h)\{n-1\}
\end{equation*}
where $h$ has degree $-4$ and the $x_i$ are the variables associated to one edge on each component of $D$.
\end{prop}
\begin{proof}
Using Proposition~\ref{slideswitch} and~\ref{arceliminate}, the dotted arcs in the diagram can be arranged so that the closed components form a sequence $K_1,\dots,K_n$ and the dotted arcs connect $K_i$ to $K_{i+1}$. Then $R(D)$ is a polynomial ring $\Z[x_1,\dots,x_n]$ where $x_i$ is the variable associated to an edge on $K_i$ (with some chosen orientation), and $C(D)$ is the Koszul factorization
\begin{equation*}
K\left((0,\dots,0),(x_1^2-x_2^2,\dots,x_{n-1}^2 - x_n^2)\right)
\end{equation*}
Since the sequence $x_1^2-x_2^2,\dots,x_{n-1}^2 - x_n^2$ is regular, the homology of this complex is the quotient $R(D)/(x_i^2 - x_j^2)$ with the grading shifted since these terms appear in the second part of the Koszul factorization rather than the first.
\end{proof}
This is exactly the module assigned to $n$ circles by the Frobenius extension $\mathcal{F}_3$ from \cite{khovanov2004link}. This Frobenius extension also controls the action of the saddle maps on homology.

Suppose that two decorated tangle diagrams $D$,$D'$ are identical outside some region in which $D$ looks like the basic tangle $D_0$ from Figure~\ref{basics}, and $D'$ looks like $D_1$. Then $R(D) = R(D')$, and there is a complex $C_{\text{outside}}$ formed by all the crossings and arcs away from the changed region such that $$C(D) = C_{\text{outside}} \otimes_{R(D)} C(D_0)$$ and $$C(D') = C_{\text{outside}} \otimes_{R(D)} C(D_1).$$ 

\begin{defn}\label{saddlemap}
The saddle map $s_{D \to D'}: C(D) \to C(D')\{1\}$ is the tensor product of the elementary saddle map $s_{D_0 \to D_1}$ used in defining $C(D)$ with the identity on $C_{\text{outside}}$.
\end{defn}

\begin{lemma}\label{saddleplanar}
If $D$ and $D'$ are two planar, boundaryless diagrams differing by a saddle move, then the map on homology induced by $s_{D_0 \to D_1}$ is the same as the cobordism map coming from the $\mathcal{F}_3$ Frobenius extension.
\end{lemma}
\begin{proof}
Again by moving the dotted arcs, it suffices to consider the saddle maps looking locally like the two rows of Figure~\ref{R1}. In the first row, the merge map is given by the horizontal arrows here:
\begin{equation*}
\xymatrix{
\Z[x_1,x_2] \ar[d]^{x_1^2 - x_2^2} \ar[r]^{x_1+x_2} & \Z[x_1,x_2] \ar[d]^{x_1-x_2} \\
\Z[x_1,x_2] \ar[r]^1 & \Z[x_1,x_2]
}
\end{equation*}
where $x_1$ and $x_2$ are the variables associated to the edges meeting the boundary and the internal edge, respectively. Embedding this diagram in a larger closed diagram and taking homology, this becomes just the quotient map identifying $x_1$ and $x_2$, which is the merge map for $\mathcal{F}_3$. Similarly, the split map in the second row of Figure~\ref{R1} is given by these horizontal arrows
\begin{equation*}
\xymatrix{
\Z[x_1,x_2] \ar[d]^{x_1 - x_2} \ar[r]^{1} & \Z[x_1,x_2] \ar[d]^{x_1^2-x_2^2} \\
\Z[x_1,x_2] \ar[r]^{x_1+x_2} & \Z[x_1,x_2]
}
\end{equation*}
On homology, this is similarly the $\mathcal{F}_3$ split map, sending an element $$x^n \in \Z[x] = \Z[x_1,x_2]/(x_1-x_2)$$ to $$x_1^n(x_1+x_2) = x_2^n(x_1+x_2) \in \Z[x_1,x_2]/(x_1^2-x_2^2).$$ 
\end{proof}

These moves also suffice to express $C(D)$ for any decorated planar diagram $D$ as a direct sum of complexes $C(D_i)$ where each $D_i$ is a diagram in which every region meets the boundary. 

\begin{prop}\label{homspace}
Suppose $D_1$ and $D_2$ are decorated diagrams of planar tangles with $2k$ boundary points, such that when $D_1$ is glued to the mirror of $D_2$ along their boundaries the resulting tangle has $c$ closed components. Then the group of chain maps $C(D_1) \to C(D_2)\{n\}$ modulo $0$-homotopy is zero if $n < k - c$ or if $n$ and $k-c$ have different parity, and is isomorphic to $\Z$ for $n = k-c$.
\end{prop}
\begin{proof}
In order for each region to meet the boundary, each of $D_1$ and $D_2$ must have $k$ components connected by $k-1$ dotted arcs, so $C(D_1)$ and $C(D_2)$ must be tensor products of $k-1$ copies of the basic planar tangles in the upper part of Figure~\ref{basics}. For any of the basic tangles $D_0$, the dual of $C(D_0)$ is isomorphic to $C(m(D_0))\{-1\}$ where $m(D)$ is the mirror image of a diagram, so the dual of $C(D_2)$ is isomorphic to $C(m(D_2))\{1-k\}$. Therefore, the group of homotopy classes of chain maps $C(D_1) \to C(D_2)\{n\}$ is isomorphic to the degree $0$ piece of the homology of $$C(D_1) \otimes C(m(D_2))\{1-k+n\} = C(D_1 \cup m(D_2))\{1-k+n\}.$$ The diagram $D_1 \cup m(D_2)$ has $c$ components, so by Proposition~\ref{closedplanar} the homology of $C(D_1 \cup m(D_2))$ is isomorphic to $M\{c-1\}$ where $M$ is some module supported in nonpositive even degrees and with degree $0$ piece $\Z$. The group of $0$-homotopy classes of chain maps $C(D_1) \to C(D_2)$ is then the degree $0$ piece of $M\{c-k+n\}$, which is zero unless $c-k+n$ is even and nonnegative and isomorphic to $\Z$ if $c-k+n = 0$. 
\end{proof}

This will often show that two maps are homotopic up to scale. To pin down the factor we will need a slightly more technical result:
\begin{lemma}\label{unique}
With $D_1$, $D_2$, $k$, and $c$ as above, suppose that ${f,g: C(D_1) \to C(D_2)\{k-c\}}$ are two chain maps such that, with respect to some bases for $C(D_1)$ and $C(D_2)$, both $f$ and $g$ have some matrix entry equal to $\pm 1$. Then $f$ is $0$-homotopic to $\pm g$.
\end{lemma}
\begin{proof}
By Proposition~\ref{homspace}, there are coprime integers $a,b$ such that $af$ and $bg$ are homotopic. Since the matrix entries of the differentials of $C(D_1)$ and $C(D_2)$ consist only of polynomials with no constant term, this can only be true if $af$ and $bg$ have the same constant term. Since $f$ and $g$ both have some matrix entry equal to $\pm 1$, $a$ and $b$ can only be $\pm 1$ and the result follows.
\end{proof}

\subsubsection{Reidemeister moves}

\begin{figure}[h]
  \centering
  \def\svgwidth{\columnwidth}
\begingroup%
  \makeatletter%
  \providecommand\color[2][]{%
    \errmessage{(Inkscape) Color is used for the text in Inkscape, but the package 'color.sty' is not loaded}%
    \renewcommand\color[2][]{}%
  }%
  \providecommand\transparent[1]{%
    \errmessage{(Inkscape) Transparency is used (non-zero) for the text in Inkscape, but the package 'transparent.sty' is not loaded}%
    \renewcommand\transparent[1]{}%
  }%
  \providecommand\rotatebox[2]{#2}%
  \newcommand*\fsize{\dimexpr\f@size pt\relax}%
  \newcommand*\lineheight[1]{\fontsize{\fsize}{#1\fsize}\selectfont}%
  \ifx\svgwidth\undefined%
    \setlength{\unitlength}{921.25984252bp}%
    \ifx\svgscale\undefined%
      \relax%
    \else%
      \setlength{\unitlength}{\unitlength * \real{\svgscale}}%
    \fi%
  \else%
    \setlength{\unitlength}{\svgwidth}%
  \fi%
  \global\let\svgwidth\undefined%
  \global\let\svgscale\undefined%
  \makeatother%
  \begin{picture}(1,0.33556972)%
    \lineheight{1}%
    \setlength\tabcolsep{0pt}%
    \put(0,0){\includegraphics[width=\unitlength,page=1]{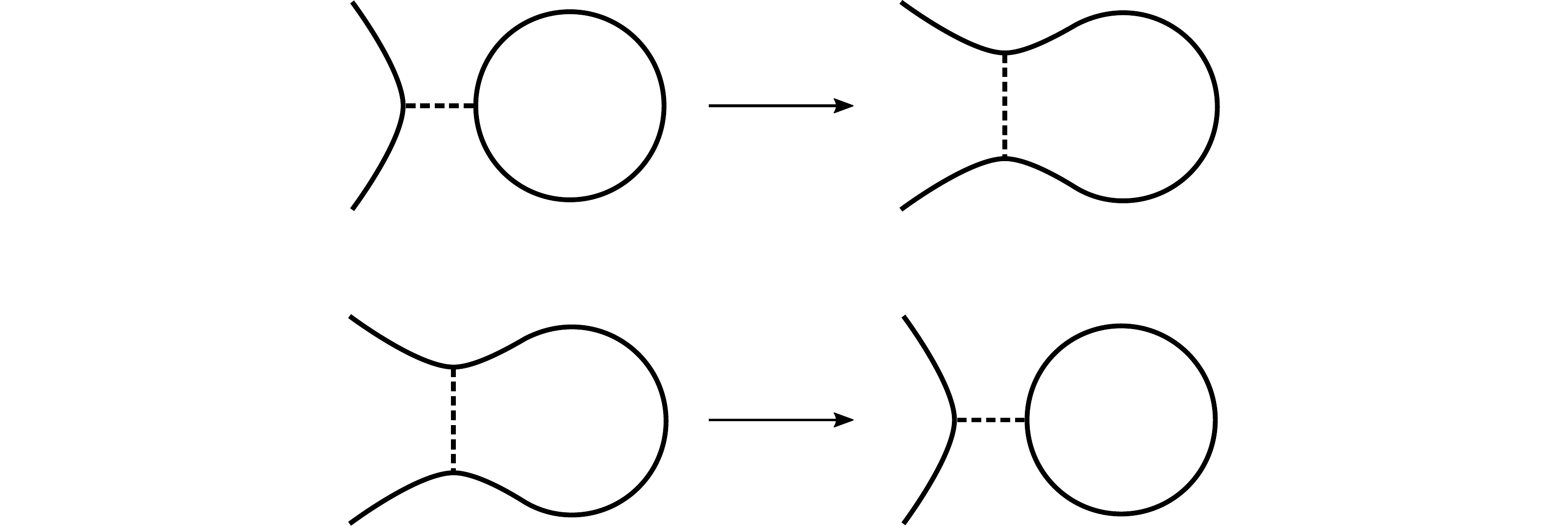}}%
    \put(0.79074484,0.30486852){\color[rgb]{0,0,0}\makebox(0,0)[lt]{\lineheight{1.25}\smash{\begin{tabular}[t]{l}$\{-1,1\}$\end{tabular}}}}%
    \put(0.78725579,0.10433936){\color[rgb]{0,0,0}\makebox(0,0)[lt]{\lineheight{1.25}\smash{\begin{tabular}[t]{l}$\{-1\}$\end{tabular}}}}%
    \put(0.30103856,0.70825113){\color[rgb]{0,0,0}\makebox(0,0)[lt]{\begin{minipage}{0.1522999\unitlength}\raggedright \end{minipage}}}%
    \put(0.42439663,0.30486852){\color[rgb]{0,0,0}\makebox(0,0)[lt]{\lineheight{1.25}\smash{\begin{tabular}[t]{l}$\{1\}$\end{tabular}}}}%
    \put(0.42090771,0.10433938){\color[rgb]{0,0,0}\makebox(0,0)[lt]{\lineheight{1.25}\smash{\begin{tabular}[t]{l}$\{1,-1\}$\end{tabular}}}}%
  \end{picture}%
\endgroup%

  \caption{The complexes for both versions of the first Reidemeister move}\label{R1}
\end{figure}
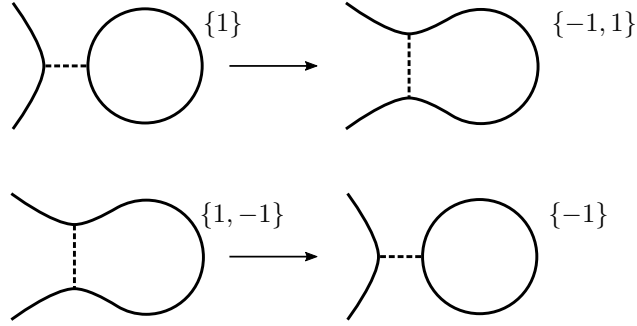

Figure~\ref{R1} shows the complexes associated to the one-crossing tangles involved in the positive and negative Reidemeister I moves, respectively. Using Proposition~\ref{arceliminate} and Proposition~\ref{loopeliminate}, the complex for the positive Reidemeister I move is $0$-homotopy equivalent to the following complex:
\begin{equation*}
\xymatrix{
\Z[x]\{2\} \ar@{}[d]|{\oplus} \ar[r]^{1} & \Z[x]\{-1,1\} \\
\Z[x] \ar[ur]_{x} &
}
\end{equation*}
In this diagram, $x$ is the variable associated to the edge at the top of the diagram. To verify that the differentials are as described, either use Lemma~\ref{unique} or compute directly. This complex is then $1$-homotopy equivalent to a single copy of $\Z[x]$, which is the complex of a crossingless arc.
Similarly, the complex for the negative Reidemeister I move is $0$-homotopy equivalent to this complex:
\begin{equation*}
\xymatrix{
\Z[x]\{1,-1\} \ar[r]^{x} \ar[dr]_{1} & \Z[x] \ar@{}[d]|{\oplus} \\
 & \Z[x]\{-2\}
}
\end{equation*}
which is again $1$-homotopy equivalent to $\Z[x]$.

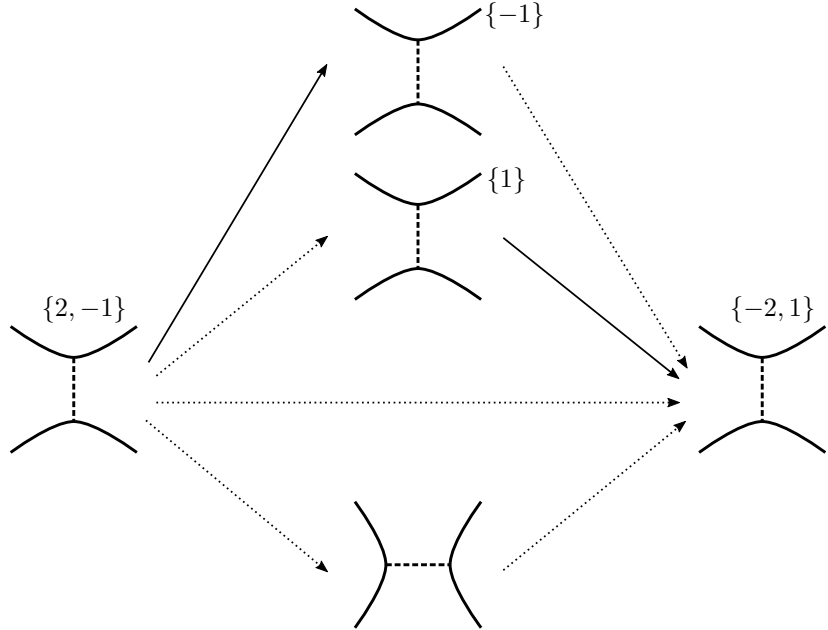
\begin{figure}[h]
  \centering
  \def\svgwidth{\columnwidth}
\begingroup%
  \makeatletter%
  \providecommand\color[2][]{%
    \errmessage{(Inkscape) Color is used for the text in Inkscape, but the package 'color.sty' is not loaded}%
    \renewcommand\color[2][]{}%
  }%
  \providecommand\transparent[1]{%
    \errmessage{(Inkscape) Transparency is used (non-zero) for the text in Inkscape, but the package 'transparent.sty' is not loaded}%
    \renewcommand\transparent[1]{}%
  }%
  \providecommand\rotatebox[2]{#2}%
  \newcommand*\fsize{\dimexpr\f@size pt\relax}%
  \newcommand*\lineheight[1]{\fontsize{\fsize}{#1\fsize}\selectfont}%
  \ifx\svgwidth\undefined%
    \setlength{\unitlength}{921.25984252bp}%
    \ifx\svgscale\undefined%
      \relax%
    \else%
      \setlength{\unitlength}{\unitlength * \real{\svgscale}}%
    \fi%
  \else%
    \setlength{\unitlength}{\svgwidth}%
  \fi%
  \global\let\svgwidth\undefined%
  \global\let\svgscale\undefined%
  \makeatother%
  \begin{picture}(1,0.65297129)%
    \lineheight{1}%
    \setlength\tabcolsep{0pt}%
    \put(0,0){\includegraphics[width=\unitlength,page=1]{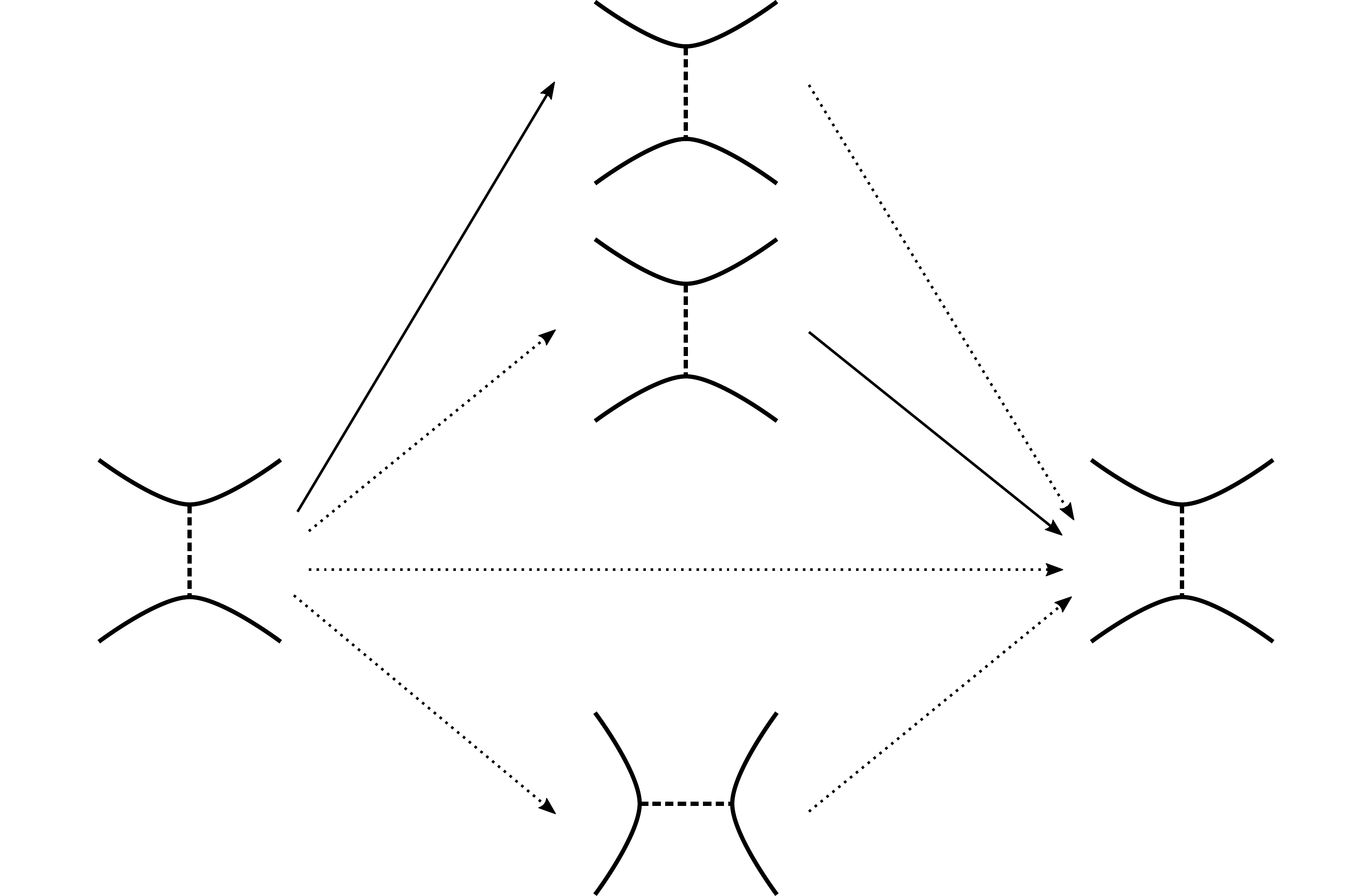}}%
    \put(0.10460009,0.32914761){\color[rgb]{0,0,0}\makebox(0,0)[lt]{\lineheight{1.25}\smash{\begin{tabular}[t]{l}$\{2,-1\}$\end{tabular}}}}%
    \put(0.5699009,0.63835787){\color[rgb]{0,0,0}\makebox(0,0)[lt]{\lineheight{1.25}\smash{\begin{tabular}[t]{l}$\{-1\}$\end{tabular}}}}%
    \put(0.57225418,0.46483827){\color[rgb]{0,0,0}\makebox(0,0)[lt]{\lineheight{1.25}\smash{\begin{tabular}[t]{l}$\{1\}$\end{tabular}}}}%
    \put(0.82798839,0.32914761){\color[rgb]{0,0,0}\makebox(0,0)[lt]{\lineheight{1.25}\smash{\begin{tabular}[t]{l}$\{-2,1\}$\end{tabular}}}}%
  \end{picture}%
\endgroup%

  \caption{The $0$-simplified complex for the second Reidemeister move. The vertical factorization is the direct sum of the shown planar tangle factorizations, and the arrows represent differentials that increase the filtration grading.}\label{R2}
\end{figure}

For the Reidemeister II move, beginning with the complex associated to the two-crossing tangle involved and simplifying using Propositions~\ref{arceliminate},~\ref{slideswitch},~and~\ref{loopeliminate} produces a complex of the shape shown in Figure~\ref{R2}. Using Lemma~\ref{unique}, the maps represented by solid arrows in that diagram can be shown to be homotopic (up to sign) to the identity, so the whole complex is $0$-homotopy equivalent to one in which those maps are $\pm 1$. The resulting complex is then $1$-homotopy equivalent to the bottom summand of the central column, which is the other side of the second Reidemeister move. 

\begin{figure}[h]
  \centering
  \def\svgwidth{\columnwidth}
\begingroup%
  \makeatletter%
  \providecommand\color[2][]{%
    \errmessage{(Inkscape) Color is used for the text in Inkscape, but the package 'color.sty' is not loaded}%
    \renewcommand\color[2][]{}%
  }%
  \providecommand\transparent[1]{%
    \errmessage{(Inkscape) Transparency is used (non-zero) for the text in Inkscape, but the package 'transparent.sty' is not loaded}%
    \renewcommand\transparent[1]{}%
  }%
  \providecommand\rotatebox[2]{#2}%
  \newcommand*\fsize{\dimexpr\f@size pt\relax}%
  \newcommand*\lineheight[1]{\fontsize{\fsize}{#1\fsize}\selectfont}%
  \ifx\svgwidth\undefined%
    \setlength{\unitlength}{892.91338583bp}%
    \ifx\svgscale\undefined%
      \relax%
    \else%
      \setlength{\unitlength}{\unitlength * \real{\svgscale}}%
    \fi%
  \else%
    \setlength{\unitlength}{\svgwidth}%
  \fi%
  \global\let\svgwidth\undefined%
  \global\let\svgscale\undefined%
  \makeatother%
  \begin{picture}(1,0.21288581)%
    \lineheight{1}%
    \setlength\tabcolsep{0pt}%
    \put(0,0){\includegraphics[width=\unitlength,page=1]{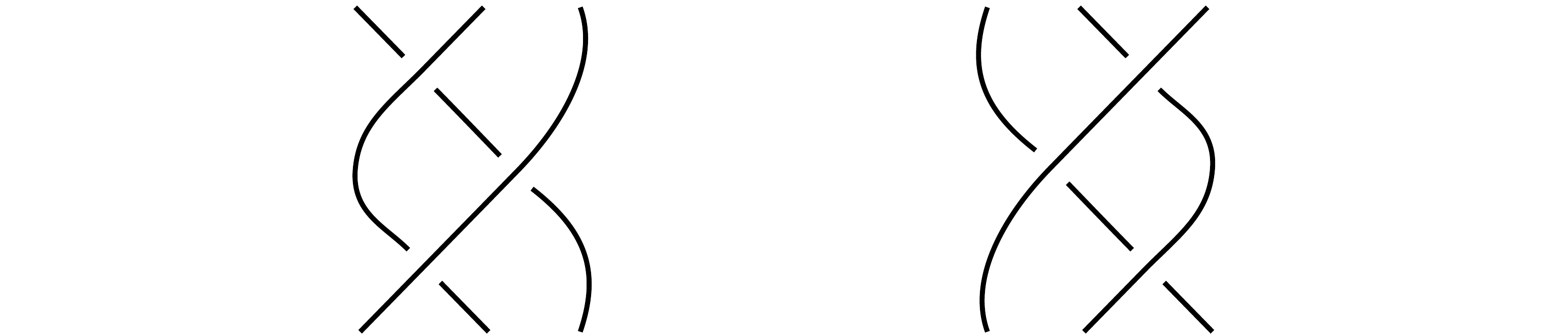}}%
    \put(0.15839003,0.10292235){\color[rgb]{0,0,0}\makebox(0,0)[lt]{\lineheight{1.25}\smash{\begin{tabular}[t]{l}$D_1$\end{tabular}}}}%
    \put(0.57332395,0.10292235){\color[rgb]{0,0,0}\makebox(0,0)[lt]{\lineheight{1.25}\smash{\begin{tabular}[t]{l}$D_2$\end{tabular}}}}%
  \end{picture}%
\endgroup%

  \caption{The tangles involved in the third Reidemeister move}\label{R3_tangle}
\end{figure}

Finally, we need to show that $C(D_1)$ and $C(D_2)$ are isomorphic for the two tangle diagrams $D_1$ and $D_2$ depicted in Figure~\ref{R3_tangle}, the two sides of the Reidemeister III move. For $D_1$, the full complex $C(D_1)$ is shown in Figure~\ref{R3_orig}. The shaded subcomplex is isomorphic to a tensor product of the complex used in the proof of Reidemeister II invariance and one of the basic one-arc complexes, so is $1$-homotopy equivalent to it's upper right entry. Applying this isomorphism and Proposition~\ref{arceliminate}, the whole complex is $1$-homotopy equivalent to a complex as shown in Figure~\ref{R3_simple}. Call these simplified complexes $C(D_1)'$ and $C(D_2)'$. These have the same vertical factorization and are supported in three adjacent filtration gradings, so it suffices to make the higher differentials $d_1$ and $d_2$ are equal. 

The entries of $d_1$ will be chain maps in the places shown by solid arrows in Figure~\ref{R3_simple}. For each of these pairs of planar diagrams, gluing one to the mirror of the other produces two closed components, so since $d_1$ has degree $-3$ with respect to the internal grading Proposition~\ref{homspace} shows that the group of such chain maps up to homotopy is isomorphic to $\Z$. Each of these maps in $C(D_1)'$ or $C(D_2)'$ is conjugate to a saddle morphism by isomorphisms, so Lemma~\ref{unique} applies to show that the entries of $d_1$ in $C(D_1)'$ and $C(D_2)'$ are homotopic up to sign. Therefore, by Lemma~\ref{dmodify}, $C(D_2)'$ is $0$-homotopy equivalent to a complex $C(D_2)''$ that has $d_0$ and $d_1$ identical to $C(D_1)'$.

Now that the two complexes have $d_0$ and $d_1$ identical, Lemma~\ref{dmodify} shows that the difference between the respective $d_2$'s in $C(D_1)'$ and $C(D_2)''$ must be a chain map along the dotted arrows in Figure~\ref{R3_simple}. Gluing the planar diagram in the leftmost column to the mirror of either of the diagrams in the rightmost column gives one closed component, so Proposition~\ref{homspace} shows that any chain map of degree $-3$ following one of the dotted arrows is nullhomotopic. Therefore, $C(D_1)'$ and $C(D_2)''$ are $0$-homotopy equivalent by another application of Lemma~\ref{dmodify}, so the original complexes $C(D_1)$ and $C(D_2)$ are $1$-homotopy equivalent.

\begin{figure}[h]
  \centering
  \def\svgwidth{\columnwidth}
\begingroup%
  \makeatletter%
  \providecommand\color[2][]{%
    \errmessage{(Inkscape) Color is used for the text in Inkscape, but the package 'color.sty' is not loaded}%
    \renewcommand\color[2][]{}%
  }%
  \providecommand\transparent[1]{%
    \errmessage{(Inkscape) Transparency is used (non-zero) for the text in Inkscape, but the package 'transparent.sty' is not loaded}%
    \renewcommand\transparent[1]{}%
  }%
  \providecommand\rotatebox[2]{#2}%
  \newcommand*\fsize{\dimexpr\f@size pt\relax}%
  \newcommand*\lineheight[1]{\fontsize{\fsize}{#1\fsize}\selectfont}%
  \ifx\svgwidth\undefined%
    \setlength{\unitlength}{892.91338583bp}%
    \ifx\svgscale\undefined%
      \relax%
    \else%
      \setlength{\unitlength}{\unitlength * \real{\svgscale}}%
    \fi%
  \else%
    \setlength{\unitlength}{\svgwidth}%
  \fi%
  \global\let\svgwidth\undefined%
  \global\let\svgscale\undefined%
  \makeatother%
  \begin{picture}(1,0.77952111)%
    \lineheight{1}%
    \setlength\tabcolsep{0pt}%
    \put(0,0){\includegraphics[width=\unitlength,page=1]{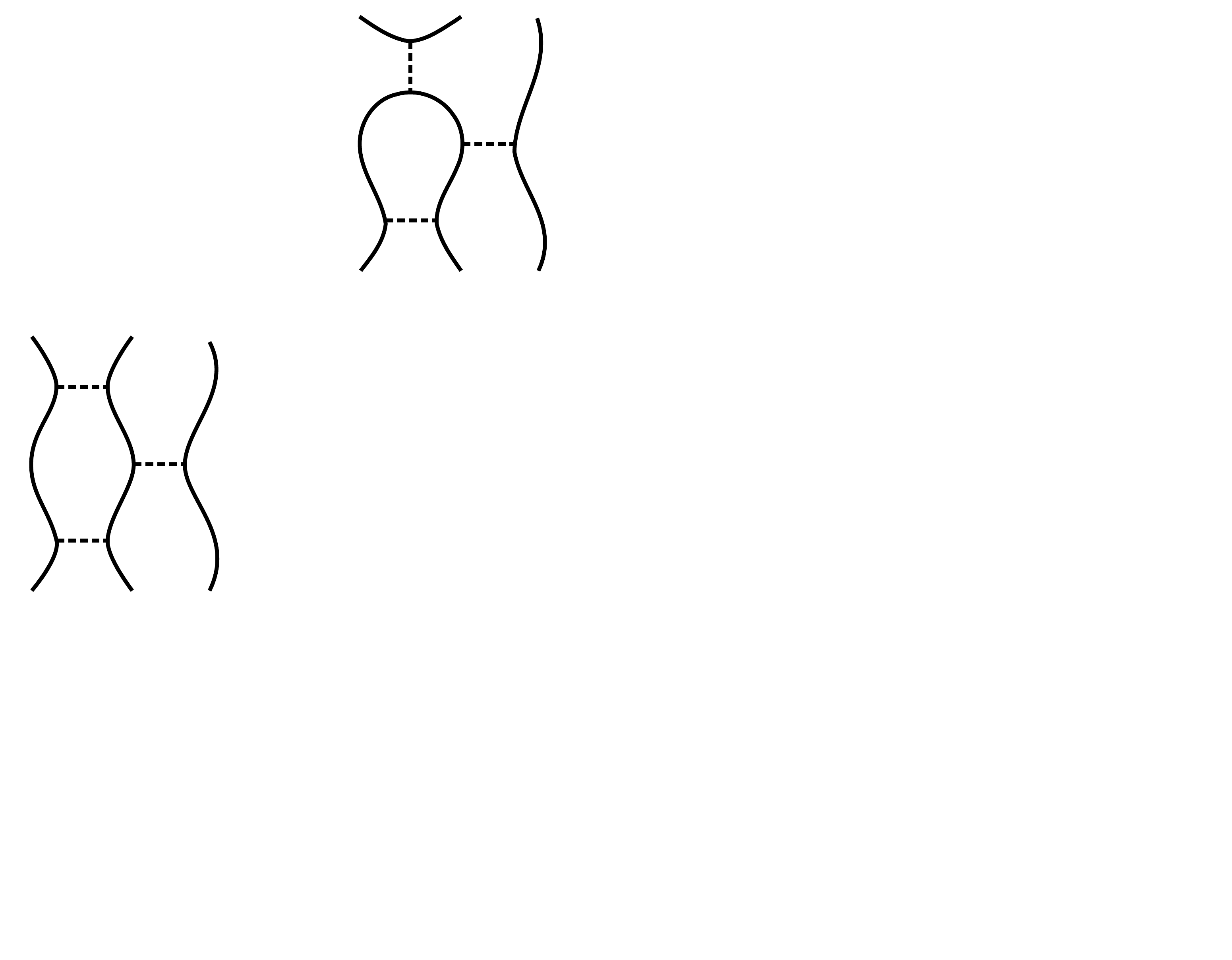}}%
    \put(0.30958046,0.01103933){\color[rgb]{0,0,0}\makebox(0,0)[lt]{\lineheight{1.25}\smash{\begin{tabular}[t]{l}$\{1,1\}$\end{tabular}}}}%
    \put(0.58496301,0.01103933){\color[rgb]{0,0,0}\makebox(0,0)[lt]{\lineheight{1.25}\smash{\begin{tabular}[t]{l}$\{-1,2\}$\end{tabular}}}}%
    \put(0.86034574,0.01103933){\color[rgb]{0,0,0}\makebox(0,0)[lt]{\lineheight{1.25}\smash{\begin{tabular}[t]{l}$\{-3,3\}$\end{tabular}}}}%
    \put(0.03407188,0.01103933){\color[rgb]{0,0,0}\makebox(0,0)[lt]{\lineheight{1.25}\smash{\begin{tabular}[t]{l}$\{3\}$\end{tabular}}}}%
    \put(0,0){\includegraphics[width=\unitlength,page=2]{R3_orig.pdf}}%
  \end{picture}%
\endgroup%

  \caption{Original complex for one side of the third Reidemeister move. To get the other side, rotate each planar tangle $180^\circ$. The shifts in the filtration grading assume that all strands are oriented downwards; changing orientations just changes them by an overall constant.}\label{R3_orig}
\end{figure}

\begin{figure}[h]
  \centering
  \def\svgwidth{\columnwidth}
\begingroup%
  \makeatletter%
  \providecommand\color[2][]{%
    \errmessage{(Inkscape) Color is used for the text in Inkscape, but the package 'color.sty' is not loaded}%
    \renewcommand\color[2][]{}%
  }%
  \providecommand\transparent[1]{%
    \errmessage{(Inkscape) Transparency is used (non-zero) for the text in Inkscape, but the package 'transparent.sty' is not loaded}%
    \renewcommand\transparent[1]{}%
  }%
  \providecommand\rotatebox[2]{#2}%
  \newcommand*\fsize{\dimexpr\f@size pt\relax}%
  \newcommand*\lineheight[1]{\fontsize{\fsize}{#1\fsize}\selectfont}%
  \ifx\svgwidth\undefined%
    \setlength{\unitlength}{892.91338583bp}%
    \ifx\svgscale\undefined%
      \relax%
    \else%
      \setlength{\unitlength}{\unitlength * \real{\svgscale}}%
    \fi%
  \else%
    \setlength{\unitlength}{\svgwidth}%
  \fi%
  \global\let\svgwidth\undefined%
  \global\let\svgscale\undefined%
  \makeatother%
  \begin{picture}(1,0.59698782)%
    \lineheight{1}%
    \setlength\tabcolsep{0pt}%
    \put(0,0){\includegraphics[width=\unitlength,page=1]{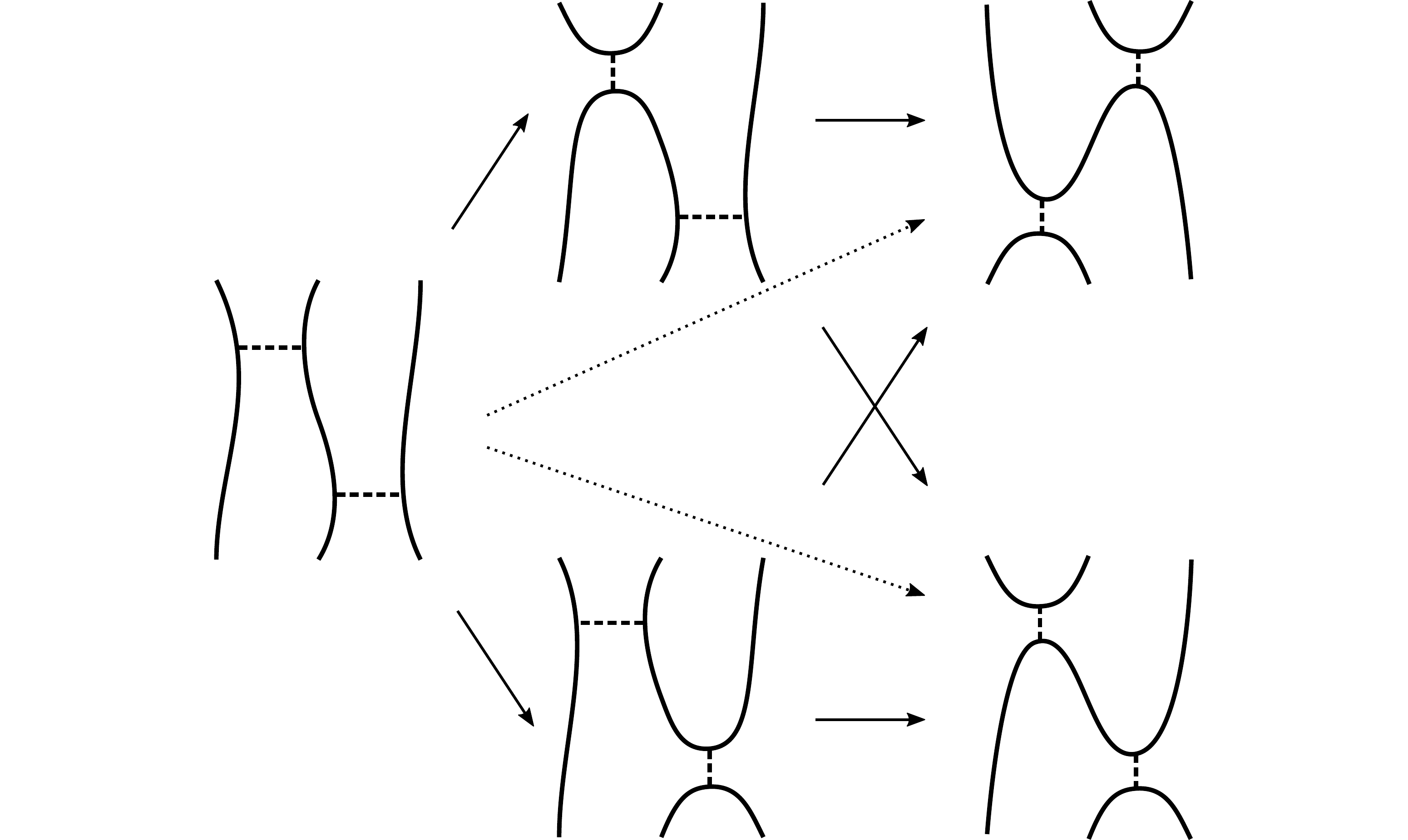}}%
    \put(0.55318722,0.56955443){\color[rgb]{0,0,0}\makebox(0,0)[lt]{\lineheight{1.25}\smash{\begin{tabular}[t]{l}$\{1,1\}$\end{tabular}}}}%
    \put(0.55318722,0.01038969){\color[rgb]{0,0,0}\makebox(0,0)[lt]{\lineheight{1.25}\smash{\begin{tabular}[t]{l}$\{1,1\}$\end{tabular}}}}%
    \put(0.8567681,0.57075437){\color[rgb]{0,0,0}\makebox(0,0)[lt]{\lineheight{1.25}\smash{\begin{tabular}[t]{l}$\{-1,2\}$\end{tabular}}}}%
    \put(0.8567681,0.01758922){\color[rgb]{0,0,0}\makebox(0,0)[lt]{\lineheight{1.25}\smash{\begin{tabular}[t]{l}$\{-1,2\}$\end{tabular}}}}%
    \put(0.30843072,0.36916705){\color[rgb]{0,0,0}\makebox(0,0)[lt]{\lineheight{1.25}\smash{\begin{tabular}[t]{l}$\{3\}$\end{tabular}}}}%
  \end{picture}%
\endgroup%

  \caption{The shape of the simplified complex for the third Reidemeister move.}\label{R3_simple}
\end{figure}

\section{Knot and link invariants}

Let $K$ be an oriented link in $S^3$, and $D$ a diagram for $K$. By Proposition~\ref{tangleinvariant}, the multifactorization $C(K) = C(D)\{0,-n_-\}$ is an invariant of $K$ up to $1$-homotopy equivalence. Since $K$ is a tangle with empty boundary, the potential of this multifactorization is zero and the boundary ring is just $\Z$, so $C(K)$ can be alternatively seen as a filtered chain complex over $\Z$. In fact, $C(K)$ carries an invariant action of a larger ring.

\begin{lemma}\label{linkmodule}
If $K$ has $n$ components, the ring $\Z[x_1,\dots,x_n]$ acts on $C(K)$ (with each variable corresponding to a component of $K$). This action commutes with with the isomorphisms coming from Reidemeister moves up to $1$-homotopy, and the action of each element $x_i^2-x_j^2$ is $1$-nullhomotopic.
\end{lemma}
\begin{proof}
In a diagram $D$ of $K$, choose a collection of $n$ basepoints $p_1,\dots,p_n$ with one on each component. For each $p_i$, let $r_i$ and $\ell_i$ be the regions adjacent to $p_i$ on the right and left according to the orientation of $K$, and let $x_i$ act on $C(K)$ by multiplication by $x_{r_i} - x-{\ell_i}$ (this is exactly the element $x_e$ discussed in the definition of $R(D)$, for $e$ the edge containing the basepoint oriented according to the orientation of $K$). 

By Lemma~\ref{variablehomot} (below), the action of $x_i$ is independent of the choice of basepoint on a component up to $1$-homotopy. If a Reidemeister move is performed away from the basepoints, the homotopy equivalence acts on a different tensor factor than the action of any $x_i$, so commutes with the action. By moving the basepoints, they can always be kept away from a Reidemeister move, so the action is an invariant of the knot.

To see that each $x_i^2-x_j^2$ acts nullhomotopically, first use Reidemeister moves to put a strand from the $i$th component adjacent to a strand from the $j$th, and move the basepoints on those components to the parallel pair of strands. Using Lemma~\ref{arceliminate}, a dotted arc can be inserted between these strands, and then the second part of Lemma~\ref{koszulhomot} gives the desired homotopy.
\end{proof}

\begin{lemma}\label{variablehomot}
Suppose that $e$ and $f$ are two oriented points on the same strand of a tangle diagram $D$ with compatible orientations. Then multiplication by $x_e$ and $x_f$ give $1$-homotopic endomorphisms of $C(D)$.
\end{lemma}
\begin{proof}
It suffices to prove this for each of the basic tangles in Figure~\ref{basics}. For $D_0$ and $D_1$, any difference $x_e-x_f$ as in the statement is equal to $\pm x_0 \mp x_2$ or $\pm x_1 \mp x_3$, respectively, so is nullhomotopic by Lemma~\ref{koszulhomot}. 
For the crossings, the difference $x_e-x_f$ will equal $\pm(x_0-x_1+x_2-x_3)$. In this case, the complex is the mapping cone of a saddle map, and the appropriate saddle map going the other direction will be a nullhomotopy of $x_0-x_1+x_2-x_3$. This is essentially the same homotopy defined in \cite{hedden2013khovanov} in the context of ordinary Khovanov homology. 
\end{proof}

If $K$ additionally has one of its components chosen, we can define a reduced homology.

\begin{defn}
The reduced complex associated to $K$, $\bar{C}(K,i)$, is the quotient $C(K)/x_i$. 
\end{defn}

When the basepoint is implicit, for instance when $K$ is a knot, we will write just $\bar{C}(K)$. In what follows, both $C(K)$ and $\bar{C}(K)$ will be treated as filtered chain complexes, the underlying unfiltered complexes will be call the total complexes, and the homology of the total complex the total homology. Also, changing the overall orientation of $K$ doesn't change the complexes $C(K)$ and $\bar{C}(K)$ and only changes the sign of the action of each $x_i$, so these can be viewed as invariants of relatively oriented links (e.g. unoriented knots) when these signs are unimportant.

\subsection{Cobordism maps}

\subsubsection{Crossing changes and nonorientable bands}

Suppose that $T_+$ and $T_-$ are tangles with diagrams that are identical except for one crossing change, in which $T_+$ has a positive crossing and $T_-$ has a negative crossing. Then $C(T_+)$ and $C(T_-)$ can both be decomposed as a tensor product of a common factor coming from all crossings except the changed one with the factorization associated to a positive or negative crossing, respectively. In the following diagram, the first row is the factorization of a negative crossing, the second row is the factorization of a positive crossing, and the vertical and diagonal arrows (after tensoring with the identity on the common factor) represent a linear map $c_+: C(T_-) \to C(T_+)$:
\begin{equation}\label{poscrossingchange}
\xymatrix@C=60pt@R=40pt{
C(D_1)\{1,-1\} \ar[r]^{s_{D_1 \to D_0}} \ar[dr]_{I_{10}} & C(D_0)\{-1\} \ar[d]^{x_0-x_1-x_2+x_3} \ar[dr]^{I_{01}} & \\
 & C(D_0)\{1\} \ar[r]_{s_{D_0 \to D_1}} & C(D_1)\{-1,1\}
}
\end{equation}

In the above diagram, the maps $I_{01}$ and $I_{10}$ are given by the downwards and upward pointing arrows here, respectively:
\begin{equation*}
\xymatrix@C=100pt@R=40pt{
R\{1\} \ar[r]^{x_0-x_2} \ar@/^/[d]^{1}  & R \ar[r]^{(x_1-x_3)(x_1-x_2+x_3-x_0)} \ar@/^/[d]^{1} & \\
R\{1\} \ar[r]^{x_1-x_3} \ar@/^/[u]^{1} & R \ar[r]^{(x_0-x_2)(x_1-x_2+x_3-x_0)}  \ar@/^/[u]^{1} & \\
}
\end{equation*}

In the other direction,	a very similar formula defines a linear map ${c_-: C(T_+) \to C(T_-)}$:
\begin{equation}\label{negcrossingchange}
\xymatrix@C=60pt@R=40pt{
 & C(D_0)\{1\} \ar[r]_{s_{D_0 \to D_1}} \ar[dl]_{I_{01}} & C(D_1)\{-1,1\} \ar[dll]^{-x_0+x_1+x_2-x_3} \ar[dl]^{I_{10}} \\
C(D_1)\{1,-1\} \ar[r]^{s_{D_1 \to D_0}} & C(D_0)\{-1\} &
}
\end{equation}
As written, this is not a filtered map, but since it decreases the filtration grading by at most two it does give a filtered map $c_-: C(T_+) \to C(T_-)\{0,2\}$.

\begin{lemma}
The maps $I_{01}$ and $I_{10}$ satisfy the identities
\begin{align*}
I_{01} d_0 - d_1 I_{01} &= (-x_0+x_1+x_2-x_3) s_{D_0 \to D_1} \\
I_{10} d_1 - d_0 I_{10} &= (x_0-x_1-x_2+x_3) s_{D_1 \to D_0} 
\end{align*}
where $d_0$ and $d_1$ are the differentials of $C(D_0)$ and $C(D_1)$, respectively.
\end{lemma}

\begin{cor}
The maps $c_{\pm}$ defined in equations~\ref{poscrossingchange} and \ref{negcrossingchange} are chain maps and mutually inverse isomorphisms.
\end{cor}

\begin{cor}\label{linkhomology}
For any link $K$ with $c$ components, the total homology of $C(K)$ is isomorphic to $\Z[x_1,\dots,x_c]/(x_i^2-x_j^2)$. 
\end{cor}
\begin{proof}
An appropriate sequence of crossing changes transforms $K$ into an unlink $U$, and the corresponding composite of maps $c_{\pm}$ induces an isomorphism between the total homology of $C(K)$ and the total homology of the unlink $U$, which was computed in Proposition~\ref{closedplanar}. 
\end{proof}
Note that, a priori, the identification of the total homology of $C(K)$ with that of an unlink could depend on the choice of unknotting sequence. However, the action described in Lemma~\ref{linkmodule} commutes with the maps $c_{\pm}$ since basepoints can be moved away from the crossing changes. Since the only graded automorphisms of $\Z[x_1,\dots,x_c]/(x_i^2-x_j^2)$ commuting with multiplication by each $x_i$ are $\pm 1$, this identification of the total homology of $C(K)$ with $\Z[x_1,\dots,x_c]/(x_i^2-x_j^2)$ is in fact canonical up to sign.

The maps $I$ above can also be used to define isomorphisms between the reduced complexes attached to knots differing by a nonorientable band move, the situation shown in Figure~\ref{nonorientedband}. The key idea is to use homotopies among the variables $x_i$ to make $I$ into a chain map.

\begin{figure}[h]
  \centering
  \def\svgwidth{\columnwidth}
  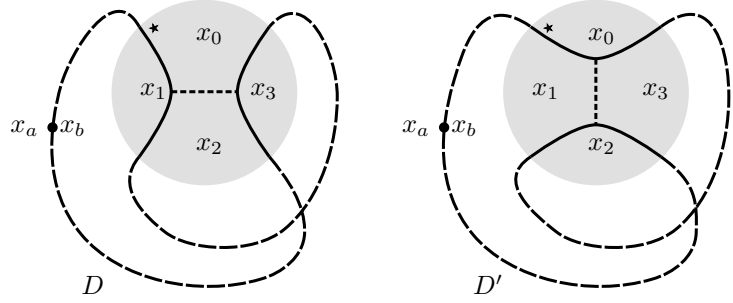
  \caption{Diagrams for two knots related by a nonorientable band move. The diagrams are as shown inside the shaded region and identical outside it, with the shared outer tangle being any tangle that has strands connecting the endpoints linked by dashed lines. The heavy dot represents the basepoint, and $x_a$ and $x_b$ the labels of the regions adjacent to it.}\label{nonorientedband}
\end{figure}

\begin{prop} \label{nonorientedmap}
Suppose that $K_0$ and $K_1$ are two knots differing by a nonorientable band move as in Figure~\ref{nonorientedband}. Then there is a chain map $$\sigma_{K_0 \to K_1}: \bar{C}(K_0) \to \bar{C}(K_1)\{0,e/2+1\},$$ where $e$ is the normal euler number of the band. Both composites ${\sigma_{K_0 \to K_1} \circ \sigma_{K_1 \to K_0}}$ and ${\sigma_{K_1 \to K_0} \circ \sigma_{K_0 \to K_1}}$ are equal to the identity.
\end{prop}
\begin{proof}
The diagrams for $K_0$ and $K_1$ can be respectively decomposed as unions $D_0 \cup D_{\text{out}}$ and $D_1 \cup D'_{\text{out}}$, where $D'_{\text{out}}$ differs from $D_{\text{out}}$ only in the relative orientation of the strands. Since orientation only comes into the definition of $C(K)$ in fixing the filtration grading, $C(K_0) = C(D_0) \otimes C(D_\text{out})$ and ${C(K_0) = C(D_0) \otimes C(D_\text{out})\{0,n_-^1 - n_-^0\}}$ where $n_-^0$ and $n_-^1$ are the numbers of negative crossings in the given diagrams of $K_0$ and $K_1$. Since these diagrams have the same number of crossings, $n_-^1 - n_-^0$ is equal to $\frac12 \left( w^0 - w^1 \right)$ where $w^i$ is the writhe of the given diagram of $K^i$. By Lemma 4.2 from \cite{ozsvath2016unoriented}, then, this is half of the normal euler number of the band.

Now, by Lemma~\ref{variablehomot}, there is a homotopy $h: D_{\text{out}} \to D_{\text{out}}\{-1.1\}$ between $x_0-x_1-x_2+x_3$ and $-2x_a+2x_b$. On the reduced complex, $x_a-x_b = 0$, so $h$ is a homotopy between $x_0-x1-x_2+x_3$ and $0$. Consider the map ${\sigma: \bar{C}(K_0) \to \bar{C}(K_1)\{0,e/2+1\}}$ defined by 
\begin{equation}
\sigma = I_{01} \otimes 1 + s_{D_0 \to D_1} \otimes h
\end{equation}
This is a chain map since, writing $d_0,d_1,$ and $d_{\text{out}}$ for the differentials of $C(D_0)$, $C(D_1)$ and $C(D_{\text{out}})$, respectively, and $s$ for $s_{D_0 \to D_1}$,
\begin{align*}
\sigma \circ (d_0 \otimes 1 + 1 \otimes d_{\text{out}}) &= I_{01} d_0 \otimes 1 + I_{01} \otimes d_{\text{out}} - s d_0 \otimes h + s \otimes h d_{\text{out}} \\
&= d_1 \circ I_{01} \otimes 1 - (x_0-x_1-x_2+x_3)s \otimes 1 + I_{01} \otimes d_{\text{out}} \\
& \qquad \qquad + d_1 s \otimes h - s \otimes d_{\text{out}} h + s \otimes (x_0-x_1-x_2+x_3) \\
&= (d_1 \otimes 1 + 1 \otimes d_{\text{out}})(I_{01} \otimes 1 + s \otimes h)
\end{align*}
In the above computation, the sign changes in the first and third line come from the fact that $s,h$ and all differentials are odd maps, so commuting them past each other introduces a sign. The fact that $\sigma_{K_0 \to K_1}$ and $\sigma_{K_1 \to K_0}$ are mutually inverse isomorphisms follows from a similar computation, using the facts that $I_{01}$ and $I_{10}$ compose to the identity in either direction and that ${I_{01} \circ s_{K_1 \to K_0} = s_{K_0 \to K_1} \circ I_{10}}$. 
\end{proof}

For a more thorough discussion of the topology of nonorientable cobordisms, see section 4 of \cite{ozsvath2016unoriented}. All we will need is the following normal form theorem for nonorientable cobordisms:
\begin{theorem}[Theorem 4.6 from \cite{ozsvath2016unoriented}, Theorem 1.3 from \cite{kamada1989nonorientable}]\label{nonorientednormalform}
If ${F \subset S^3 \times I}$ is a properly embedded nonorientable surface with $\partial F \cap S^3 \times \{0\} = K_0$ and ${\partial F \cap S^3 \times \{1\} = K_1}$, then there are knots $K_0'$ and $K_1'$ concordant to $K_0$ and $K_1$, respectively, such that $K_1'$ can be produced from $K_0'$ by adding $b_1(F)-1$ nonorientable bands with total euler number equal to $e(F)$. 
\end{theorem}

\subsubsection{Cobordisms and concordance invariance}

In definition~\ref{saddlemap}, we defined a map $s_{K \to K'}: C(K) \to C(K')\{1\}$ whenever $K$ and $K'$ are related by a cobordism consisting of a single one-handle, and computed the effect of this map on homology for planar cobordisms of unknots. We are now in a position to extend these results to more general knots and cobordisms. 

By Proposition~\ref{loopeliminate}, $C(K \sqcup U) \cong C(K)\{-1\} \oplus C(K)\{1\}$.
\begin{defn}
The zero-handle map $b_K: C(K) \to C(K \sqcup U)\{-1\}$ is the inclusion of the second summand above, and the two-handle map ${d_K: C(K \sqcup U) \to C(K)\{-1\}}$ is the projection onto the first summand.
\end{defn}

By composing the maps $s,b,$ and $d$, we can associate a map
\begin{equation*}
f_{\Sigma}: C(K) \to C(K')\{-\chi(\Sigma)\}
\end{equation*}
to any cobordism $\Sigma$ from $K$ to $K'$ together with a handle decomposition. The map $f_{\Sigma}$ potentially depends on the precise handle decomposition used, but the induced map on total homology turns out to not even depend on how $\Sigma$ is embedded. By Corollary~\ref{linkhomology}, the total homology of $C(K)$ is isomorphic to the module assigned to $n$ circles by the $\mathcal{F}_3$ Frobenius extension if $K$ has $n$ components. 

\begin{prop}\label{anycobord}
The map induced by $f_{\Sigma}$ on total homology is equal up to sign to the cobordism map from the $\mathcal{F}_3$ Frobenius extension.
\end{prop}
\begin{proof}
Since $f_{\Sigma}$ and the $\mathcal{F}_3$ cobordisms maps are defined as a composite of maps assigned to cobordisms built from a single handle, it suffices to show this for such cobordisms, and for a zero-handle or two-handle this statement follows directly from the definitions of $b_K$ and $d_K$. Therefore, suppose that $K$ and $K'$ differ by a single oriented band, and let $D$ and $D'$ be diagrams for $K$ and $K'$ differing only in a disc in which $D$ looks like the basic tangle diagram $D_0$ from Figure~\ref{basics} and $D'$ looks like $D_1$, so that the saddle map $s_{K \to K'}$ acts by the identity on the tensor factor corresponding to the fixed part of the diagram and by the elementary saddle map $s_{D_0 \to D_1}$ in the changing region. 

By a sequence of crossing changes and Reidemeister moves in the fixed region, all of the crossings can be eliminated. This transforms $D$ and $D'$ into two planar diagrams differing by a planar saddle move, and the resulting sequence of isomorphisms from Reidemeister moves and maps $c_{\pm}$ identifies the total homologies of $C(K)$ and $C(K')$ with the homologies of appropriate unlinks. Since the moves all occur in the fixed region, these isomorphisms act on a different tensor factor than the saddle map $s_{K \to K'}$, so the isomorphisms and saddle map commute. Therefore, the saddle map $s_{K \to K'}$ has the same action on total homology as the planar saddle, which was computed in Lemma~\ref{saddleplanar} to be the saddle map of the $\mathcal{F}_3$ Frobenius system.
\end{proof}

Two special cases of this computation will be crucial later:
\begin{cor}
If both $K$ and $K'$ are knots and $\Sigma$ has genus $g$, the map induced by $f_{\Sigma}$ on total homology is multiplication by $(2x)^g$. In particular, when $\Sigma$ is a concordance, $f_{\Sigma}$ induces an isomorphism on homology.
\end{cor}

\begin{cor}\label{splitinj}
If $g(\Sigma) = 0$, then if $K$ has one component $f_{\Sigma}$ is injective on total homology and if $K'$ has one component $f_{\Sigma}$ is surjective on total homology.
\end{cor}
\begin{proof}
In this case $\Sigma$ is a composite of a concordance and either only splitting one-handles or only merging one-handles. Since the $\mathcal{F}_3$ split map is injective and the merge map is surjective, the result follows.
\end{proof}

\section{Properties of the invariant}

\begin{theorem}
$C(K)$ and $\bar{C}(K)$, seen up to $1$-homotopy equivalence, are invariants of $K$.If $K$ has $c$ components, the total homology of $C(K)$ is isomorphic to the module assigned to $c$ circles by the $\mathcal{F}_3$ Frobenius extension and the total homology of $\bar{C}(K)$ is the quotient of the total homology of $C(K)$ by the variable associated to the marked component. In particular, if $K$ is a knot, the homologies of $C(K)$ and $\bar{C}(K)$ are $\Z[x]$ and $\Z$, respectively.

The $E_2$ pages of the spectral sequences coming from the filtrations on $C(K)$ and $\bar{C}(K)$ are isomorphic respectively to the $\mathcal{F}_3$ Khovanov homology and the reduced Khovanov homology of $K$. 
\end{theorem}
\begin{proof}
The first part of this statement was proven in Corollary~\ref{linkhomology}.

The computation of the $E_2$ page similarly follows from Proposition~\ref{closedplanar} and Lemma~\ref{saddleplanar}. By Proposition~\ref{closedplanar}, the $E_1$ page of the spectral sequence is exactly the module underlying the appropriate Khovanov complex, and the differential on the $E_1$ page is the maps induced on homology by planar saddles, which Lemma~\ref{saddleplanar} shows is the Khovanov differential.
\end{proof}

We can now define numerical invariants of $K$ from this complex. The definition of $t$ involves a choice of field $k$, but since none of the results of this paper depend on the choice it will be suppressed from the notation. The invariant could just as well be defined over $\Z$, but would not necessarily be a concordance homomorphism.
\begin{defn}
$t(K)$ is the largest $n$ for which there is a cycle in ${\bar{C}(K)\otimes_\Z k}$ generating the total homology that can be written as a sum of homogeneous elements with filtration grading at least $n$.
\end{defn}
\begin{defn}
For $i \ge 0$, $T_i(K)$ is the largest $n$ for which there is a cycle in $C(K)$ representing the element $(2x)^i$ in the total homology that can be written as a sum of homogeneous elements with filtration grading at least $n$.
\end{defn}

The various maps defined between the complexes $C$ and $\bar{C}$ now give bounds on the invariants $T_i$ and $t$:

\begin{prop}\label{orientedcobordismbound}
If there is an orientable cobordism $\Sigma$ of genus $g$ from $K_1$ to $K_2$, then for any $i$ we have $T_i(K_1) \le T_{i+g}(K_2)$. If $g = 0$, so $K_1$ and $K_2$ are concordant, then $T_i(K_1) = T_i(K_2)$ for each $i$ and $t(K_1) = t(K_2)$.
\end{prop}
\begin{proof}
The map $f_{\Sigma}: C(K_1) \to C(K_2)$ is filtered and acts on the total homology by multiplication by $(2x)^g$, so sends a cycle $\gamma$ representing $(2x)^i$ in the total homology of $C(K_1)$ supported in filtration degrees $\ge T_i(K_1)$ to a cycle $f_{\Sigma}(\gamma)$ representing $(2x)^{i+g}$ supported in filtration degrees $\ge T_{i}(K_1)$. 

If $g = 0$, then combining this bound with the same bound coming from the concordance $\bar{\Sigma}: K_2 \to K_1$ gives equality of all $T_i$, and similarly $f_{\Sigma}$ and $f_{\bar{\Sigma}}$ are filtered isomorphisms $\bar{C}(K_1)\otimes k \cong \bar{C}(K_2)\otimes k$ so $t(K_1) = t(K_2)$. 
\end{proof}

\begin{prop}
If there is a possibly nonorientable cobordism $F$ from $K_1$ to $K_2$ with first Betti number $b+1$ and normal Euler number $e$, then $$|t(K_1) - t(K_2) - e/2| \le b.$$
\end{prop}
\begin{proof}
By Theorem~\ref{nonorientednormalform}, $F$ factors as a concordance, then $b$ nonorientable bands of total Euler number $e$, and finally a second concordance. Since $t$ is a concordance invariant, it suffices to prove the bound when $K_1$ and $K_2$ are related by a sequence of $b$ nonorientable bands, and by applying the bands one at a time it suffices to prove the bound when $b = 1$ and $K_1$ and $K_2$ are related by a single band. 

By Proposition~\ref{nonorientedmap}, there is a chain map $\sigma_{K_1 \to K_2}: \bar{C}(K_1) \to \bar{C}(K_2)$ that induces an isomorphism on total homology and decreases the filtration by at most $e/2 + 1$, and looking at the band in reverse there is similarly a chain map $\sigma_{K_2 \to K_1}: \bar{C}(K_2) \to \bar{C}(K_1)$ that induces an isomorphism on total homology and decreases the filtration by at most $-e/2 + 1$.

Therefore, tensoring both of these maps by $k$, a generator of the homology of $\bar{C}(K_1) \otimes k$ supported in filtration degrees at least $t(K_1)$ is sent to a generator of the homology of $\bar{C}(K_2) \otimes k$ supported in filtration degrees at least $t(K_1)-e/2-1$, so $t(K_2) \ge t(K_1) - e/2 - 1$. Similarly, $t(K_1) \ge t(K_2)+e/2-1$, and these two bounds give the desired result. 
\end{proof}

Taking $K_1$ to be the unknot in the above two propositions proves Theorem~\ref{nonorientablegenusbound} and most of Theorem~\ref{orientedgenusbound}. Finishing the proof of Theorem~\ref{orientedgenusbound} requires a little more use of the module structure:
\begin{proof}[Proof of Theorem~\ref{orientedgenusbound}]
Multiplication by $2x$ is a filtered endomorphism of $C(K)$, so $T_{i+1}(K) \ge T_i(K)$. Taking $K_1$ to be an unknot in Proposition~\ref{orientedcobordismbound} gives that $T_{i}(K) \ge 0$ when $i \ge g_4(K)$, and taking $K_2$ to be an an unknot gives that $T_i(K) \le 0$ for any $i$. 
\end{proof}

\begin{prop}
If knots $K_+$ and $K_-$ differ by a crossing change in which $K_+$ has a positive crossing and $K_-$ has a negative crossing, then ${t(K_-) \le t(K_+) \le t(K_-)+2}$, and similarly ${T_i(K_-) \le T_i(K_+) \le T_i(K_-) + 2}$.
\end{prop}
\begin{proof}
The maps $c_+$ and $c_-$ both induce isomorphisms on the total homologies of both $C(K_{\pm})$ and $\bar{C}(K_{\pm})\otimes k$. Since $c_+$ is a filtered map and $c_-$ decreases the filtration degree by at most two, the result follows. 
\end{proof}

\section{Computations}

Recall that, under the isomorphism between the $E_2$ page of the spectral sequence for $\bar{C}(K)$ and the reduced Khovanov homology of $K$, the internal and filtration gradings considered here correspond with the gradings $q - 3h$ and $h$, where $q$ and $h$ are the standard quantum and homological gradings on Khovanov homology. In particular, the sum of the internal and filtration gradings is $q - 2h$, often called the delta grading. 

\begin{prop}
If $K$ is Khovanov homologically thin, $t(K) = s(K)$. 
\end{prop}
\begin{proof}
The $E_\infty$ page of a spectral sequence is exactly the associated graded module of the homology of the complex, so in the spectral sequence coming from $\bar{C}(K) \otimes k$, there is exactly one surviving term with internal grading $0$ and filtration grading $t(K)$, so with delta grading $t(K)$. A term of the same degree will appear on the $E_2$ page, so the reduced Khovanov homology of $K$ has at least one generator in delta grading $t(K)$. An identical argument with the Lee spectral sequence shows that the reduced Khovanov homology of $K$ has a generator in delta grading $s(K)$, so if $K$ is homologically thin $t(K) = s(K)$.
\end{proof}
In particular, when $K$ is alternating, $t(K) = s(K) = - \sigma(K)$, so the nonorientable genus bounds in Theorem~\ref{nonorientablegenusbound} are always zero. This is also the case for the bounds in \cite{batson2014nonorientable} and \cite{ozsvath2016unoriented}, which is not a coincidence.
\begin{lemma}
If $f$ is any knot invariant satisfying $|f(K) - e/2| \le b$ whenever $K$ bounds a surface with first Betti number $b$ and normal Euler number $e$, then $f(K) = \sigma(K)$ for alternating knots.
\end{lemma}
\begin{proof}
If $K$ has an alternating diagram with $n_+$ positive crossings, $n_-$ negative crossings, and $w$ white regions and $b$ black regions in the checkerboard coloring, then one of the checkerboard surfaces with boundary $K$ has first Betti number $n_+ + n_- - w + 1$ and normal Euler number $-2n_+$, and the other has first Betti number $n_+ + n_- - b + 1$ and normal Euler number $2n_-$. Therefore, $f(K)$ satisfies the bounds
\begin{equation*}
-n_+ + b -1 \le f(K) \le n_- - w + 1
\end{equation*}
Since a knot diagram has two more regions than crossings, $n_+ + n_- + 2 = b + w$, so the upper and lower bound coincide. Since, in particular, $\sigma(K)$ satisfies these bounds \cite{gordon1978signature}, $f(K) = \sigma(K)$. 
\end{proof}

For nonalternating knots, however, $t$ can diverge from both $s$ and $\sigma$. 
\begin{prop}\label{torusknots}
For positive, coprime integers $p,q$, the $t$ invariant of the torus knot $T_{p,q}$ is determined by the recurrence $T_{p,p+q} = T_{p,q} + \lfloor p^2/2 \rfloor$. 
\end{prop}
Since, by Propositions~2.1 and 2.2 from \cite{feller2017cobordisms}, the invariant $-2\upsilon(K) = -2\Upsilon_K(1)$ satisfies this same recurrence, we have $t(T_{p,q}) = -2\upsilon(T_{p,q})$. To prove Proposition~\ref{torusknots}, we will start with a more general upper bound.
\begin{lemma}\label{ttwistbound}
If $K_2$ is obtained from $K_1$ by a full twist on $n$ coherently oriented strands, $t(K_2) \le t(K_1) + \lfloor n^2 / 2 \rfloor$
\end{lemma}
\begin{proof}
There is a cobordism from $K_2$ to the disjoint union $K_1 \sqcup T_{n,n}$ formed by attaching a band splitting a loop from each of the twisted strands, so that the split off loops form the torus link $T_{n,n}$ and the untwisted knot $K_1$ remains. By Corollary~\ref{splitinj}, this gives an injection $f: \bar{C}(K_2)\otimes k \to \bar{C}(K_1 \sqcup T_{n,n}) \otimes k$ that shifts the internal grading by $n$ but preserves the filtration grading (the second reduced complex is taken with the basepoint on $K_1$). Suppose $c \in \bar{C}(K_2) \otimes k$ is a cycle generating the homology that is supported in filtration gradings at least $t(K_2)$. Then $f(c)$ is a nonzero cycle in $\bar{C}(K_1 \sqcup T_{n,n}) \otimes k$ supported in filtration gradings at least $t(K_2)$. 

The complex $\bar{C}(K_1 \sqcup T_{n,n})$ is isomorphic to a tensor product of $\bar{C}(K_1)$ with the quotient $C(T_{n,n})/x^2$. Any cycle in $\bar{C}(K_1) \otimes k$ that is nonzero in homology contains a term with filtration grading at most $t(K_1)$, so for there to be a cycle in $\bar{C}(K_1 \sqcup T_{n,n}) \otimes k$ supported in filtration gradings at least $t(K_2)$ there must be a nonzero cycle in $C(T_{n,n})/x^2 \otimes k$ with filtration grading at least $t(K_2) - t(K_1)$. However, such an element can only exist if there is an element in the $E_2$ page of the spectral sequence coming from $C(T_{n,n})/x^2 \otimes k$ in filtration grading at least $t(K_2) - t(K_1)$, so the unreduced Khovanov homology of $T_{n,n}$ has a nonzero element with homological grading at least $t(K_2) - t(K_1)$. It was computed by Sto\v{s}i\'c \cite{stovsic2009khovanov} that this homology group is supported in homological gradings between $0$ and $\lfloor n^2 / 2 \rfloor$, so $t(K_2) - t(K_1) \le \lfloor n^2 / 2 \rfloor$.
\end{proof}

Therefore,  $T_{p,p+q} \le T_{p,q} + \lfloor p^2/2 \rfloor$. To find a matching lower bound, we can use the bounds from nonorientable cobordisms. The pinch move on $T_{p,q}$ (discussed in more detail in \cite{jabuka2018nonorientable}) is the cobordism coming from taking the unoriented resolution of any of the crossings in the standard braid diagram of $T_{p,q}$. 

Each pinch move is a cobordism with $b_1 = 1$ from $T_{p,q}$ to $T_{r,s}$ for some $0<r<p$ and $0<s<q$, but neither the values $r,s$ nor the normal Euler number of the cobordism are entirely straightforward to compute. However, they all change in a predictable way under a full twist:
\begin{lemma}
If the pinch move from $T_{p,q}$ is a cobordism to $T_{r,s}$ with normal Euler number $e$, then the pinch move on $T_{p,q+p}$ is a cobordism to $T_{r,s+r}$ with normal Euler number $e - p^2 + r^2$.
\end{lemma}
\begin{proof}
Since all crossings in the original diagram of $T_{p,q}$ are positive, the normal Euler number of the pinch move is equal to $-2$ times the number of negative crossings in the diagram immediately after the pinch move. The rectangle in the diagram of $T_{p,q}$ in which the twisting will be done initially contains $p$ parallel coherently oriented strands, and after the pinch move $k = (p-r)/2$ strands will be oriented the other way. When these parallel strands are replaced by a full twist to form $T_{p,q+p}$, before the pinch move there will be no negative crossings but after the pinch move there will be a negative crossing wherever one of the reversed strands crosses one of the unreversed strands. Each pair of strands crosses twice in a full twist, so the number of new negative crossings created is therefore
\begin{equation*}
2k(p-k) = 2\cdot \frac12 (p-r) \cdot \frac12 (p+r)  = \frac12 (p^2-r^2)
\end{equation*}
and the change in normal Euler number is $-p^2 + r^2$.
\end{proof}

This fact together with the full twist inequality gives an inductive proof of the fact that the two bounds coincide:
\begin{prop}
For any $p,q > 0$ such that $T_{p,q}$ reduces to $T_{r,s}$ by a pinch move of Euler number $e$,
\begin{equation*}
t(T_{p,q}) + \lfloor p^2/2 \rfloor = t(T_{p,q+p}) = t(T_{r,s+r}) - (e-p^2+r^2)/2 - 1
\end{equation*}
\end{prop}
\begin{proof}
The left hand side of this equation is an upper bound for $t(T_{p,q+p})$ and the right hand side is a lower bound, so it suffices to show that the bounds are the same. Inductively, we may assume that
\begin{equation*}
t(T_{p,q}) = t(T_{r,s}) - e/2 - 1
\end{equation*}
and
\begin{equation*}
t(T_{r,s+r}) = t(T_{r,s}) + \lfloor r^2/2 \rfloor,
\end{equation*}
so
\begin{equation*}
t(T_{p,q}) = t(T_{r,s+r}) - \lfloor r^2/2 \rfloor - e/2 - 1
\end{equation*}
and the needed equality follows from adding $\lfloor p^2/2 \rfloor$ to both sides of this equation and the fact that $p$ and $r$ have the same parity.
\end{proof}

Combining the full twist inequality for $t$ used in the above computation with a similar bound on $s$ gives an additional corollary:
\begin{cor}
If $K$ and $K'$ differ by a full twist on $n \ge 6$ coherently oriented strands, $K$ and $K'$ cannot both be homologically thin.
\end{cor}
\begin{proof}
By Lemma~\ref{ttwistbound}, $t(K') \le t(K) + \lfloor n^2/2 \rfloor$, and an essentially identical argument using the cobordism maps on Lee homology to bound the $s$ invariant gives that $s(K') \ge s(K) + (n-1)(n-2)$. If $K$ is Khovanov homologically thin, then $s(K) = t(K)$, so
\begin{equation*}
s(K') \ge s(K) + (n-1)(n-2) = t(K) + (n-1)(n-2) \ge t(K') + (n-1)(n-2) - \lfloor n^2/2 \rfloor
\end{equation*}
For $n \ge 6$, $(n-1)(n-2) > \lfloor n^2 / 2 \rfloor$, so $s(K') > t(K')$ and $K'$ is not homologically thin. 
\end{proof}

The equality $t(K) = -2\upsilon(K)$ also holds for any knot that is thin for both Khovanov homology and knot Floer homology, such as an alternating knot. However, these two invariants are not equal on all knots. For the $2$-twisted positive Whitehead double of $T_{2,3}$, the same knot used by Hedden and Ording \cite{hedden2008ozsvath} to show that the invariants $s$ and $\tau$ are distinct, the unreduced Khovanov homology has Poincar\'e polynomial
\begin{align*}
&q^{18} t^9+2 q^{16} t^8+q^{14} t^7+q^{12} t^6+2 q^{10} t^5+q^{10} t^4+q^8 t^4+2 q^8 t^3+q^8 t^2+q^6 t^2\\
+ &q^6t + q^4 t + q^4 + 2 q^2 + 2 q^{-2}t^{-1} + t^{-2} + q^{-2}t^{-3} + q^{-4}t^{-4}
\end{align*}
The $t$ invariant of this knot $K$ will be the homological grading of the surviving generator on the $E_\infty$ page, which has internal grading $0$. This generator must come from an element on the $E_2$ page with internal grading $q-3t$ equal to zero, but in the above Poincar\'e polynomial the only such element has degree $q^6t^2$. Therefore, $t(K) = 2$. However, by Proposition~1.5 in \cite{feller2019upsilon}, $\upsilon(K) = 0$. 

\bibliographystyle{amsplain}
\bibliography{tinv.bib}

\end{document}